\documentclass[10pt,a4paper]{article}
\usepackage[T1]{fontenc}
\usepackage[utf8]{inputenc}
\usepackage{authblk}
\usepackage{amsmath,amssymb,amsthm,esint,bm}
\usepackage{mathrsfs}
\usepackage{bookmark}
\usepackage{amsmath}
\allowdisplaybreaks[4]

\newtheorem{definition}{Definition}[section]
\newtheorem{theorem}[definition]{Theorem}
\newtheorem{lemma}[definition]{Lemma}

\theoremstyle{remark}
\newtheorem{remark}[definition]{Remark}
\numberwithin{equation}{section}

\title{Nonlinear Potential Estimates for Generalized Stokes System}

\author[a]{Lingwei Ma}
\author[b]{Zhenqiu Zhang\thanks{Corresponding author.}}
\author[c]{Feng Zhou}

\affil[a]{School of Mathematical Sciences, Tianjin Normal University, Tianjin, 300387, P.~R.~China}
\affil[b]{School of Mathematical Sciences and LPMC, Nankai University, Tianjin, 300071, P.~R.~China}
\affil[c]{School of Mathematics and Statistics, Shandong Normal University, Jinan, Shandong, 250358, P.~R.~China}

\date{\today}

\usepackage{hyperref}
\begin{document}
\maketitle
\footnotetext[1]{E-mail: mlw1103@163.com (L. Ma), zqzhang@nankai.edu.cn (Z. Zhang), zhoufeng@u.nus.edu(F. Zhou).}
\begin{abstract}
In this paper, we consider the generalized stationary Stokes system with $p$\,-growth and Dini-$\operatorname{BMO}$ regular coefficients. The main purpose is to establish pointwise estimates for the shear rate and the associated pressure to such Stokes system in terms of an unconventional nonlinear Havin-Maz'ya-Wolff type potential of the nonhomogeneous term in the plane. As a consequence, a symmetric gradient $L^{\infty}$ estimate is obtained. Moreover, we derive potential estimates for the weak solution to the Stokes system without additional regularity assumptions on the coefficients in higher dimensional space.
\\

Mathematics Subject classification (2020): 35Q35; 35J92; 35B65.

Keywords: Nonlinear Stokes system; Discontinuous coefficients; Potential estimates; Havin-Maz'ya-Wolff potential. \\

\end{abstract}


\section{Introduction.}\label{section1}

Let $\Omega\subset\mathbb{R}^{n}$ ($n\geq2$) be a bounded John domain. The concern of this paper is to study the generalized nonlinear Stokes system on $\Omega$ as follows
\begin{equation}\label{model}
\left\{\begin{array}{r@{\ \ }c@{\ \ }ll}
\operatorname{div}\mathcal{A}\left(x,D {\bf u}\right)-\nabla\pi & =& \operatorname{div}{\bf F} & \mbox{in}\ \ \Omega\,, \\[0.05cm]
\operatorname{div}{\bf u}&=& 0 & \mbox{in}\ \ \Omega\,, \\[0.05cm]
{\bf u}&=& 0 & \mbox{on}\ \ \partial\Omega\,, \\[0.05cm]
\end{array}\right.
\end{equation}
where the velocity of the fluid flow ${\bf u}:\Omega\rightarrow\mathbb{R}^{n}$ and its pressure $\pi:\Omega\rightarrow\mathbb{R}$\,, the shear rate $D{\bf u}=\frac{\nabla {\bf u}+\left(\nabla{\bf u}\right)^{T}}{2}$ and the given exterior force ${\bf F}: \Omega\rightarrow\mathbb{R}^{n\times n}_{\rm sym}$\,.
The vector field $\mathcal{A}: \Omega\times\mathbb{R}^{n\times n}\rightarrow\mathbb{R}^{n\times n}$ is assumed to be a $C^{1}$\,-\,Carath\'{e}odory function and satisfy the following ellipticity and $p$\,-growth conditions
\begin{equation}\label{structural-cons}
\left\{\begin{array}{r@{\ \ }c@{\ \ }ll}
\left\langle\mathcal{A}(x,\xi)-\mathcal{A}(x,\eta),\xi-\eta\right\rangle &\geq& \nu\left(\mu^{2}+|\xi|^{2}+|\eta|^{2}\right)^{\frac{p-2}{2}}|\xi-\eta|^{2}\,,\\[0.05cm]
\left|\mathcal{A}(x,\xi)-\mathcal{A}(x,\eta)\right|&\leq& L\left(\mu^{2}+|\xi|^{2}+|\eta|^{2}\right)^{\frac{p-2}{2}}|\xi-\eta|\,,\\[0.05cm]
\mathcal{A}(x,{\bf 0}) &=& {\bf 0}\,,
\end{array}\right.
\end{equation}
for almost every $x\in\Omega$ and each $\xi,\,\eta\in\mathbb{R}^{n\times n}$\,. In \eqref{structural-cons}\,, $0<\nu\leq L$, $0\leq\mu\leq1$, $1<p<+\infty$, and $\langle\cdot,\cdot\rangle$ denotes the standard inner product in $\mathbb{R}^{n\times n}$\,.

Stokes system \eqref{model} originates from the flow of non-Newtonian fluid in the field of fluid mechanics, where the stress tensor $\mathcal{A}$ may depend nonlinearly on the shear rate $D{\bf u}$. Note that the behavior of the fluid is quite different for the cases $p>2$ and $1<p<2$. The former one describes the shear thickening fluids, while the latter corresponds to the shear thinning fluids. In this paper, we first establish pointwise estimates for the gradient of a solution and the pressure to \eqref{model} via an unconventional nonlinear Havin-Maz'ya-Wolff type potential of the nonhomogeneous term in two dimensional space for $p>2$ and $1<p<2$, simultaneously. Furthermore, we obtain potential estimates for the weak solution to \eqref{model} in higher dimensional space for the superquadratic case $p\geq 2$.

Nonlinear potential theory plays an essential role in the regularity theory of partial differential equations.
Its aim is to establish a unified approach to capture the regularity properties of solutions to various elliptic and parabolic equations in terms of the regularity properties of the given nonhomogeneous terms and coefficients, such as the celebrated Calder\'on–Zygmund estimates and Schauder estimates. An important impulse on this subject can be traced back to Kilpel\"{a}inen and Mal\'{y} \cite{KiMa, KiMa2}, who obtained pointwise estimates for solution itself to the quasilinear equations of $p$\,-Laplace type via the nonlinear Wolff type potentials of the nonhomogeneous term in the equations. The remarkable Havin-Maz'ya-Wolff potential was introduced by Maz'ya and Havin \cite{MaHa} and the relevant fundamental works were credited to Hedberg and Wolff \cite{HeWo}. After that Trudinger and Wang \cite{TrWa} provided similar potential estimates by virtue of a different proof. Shortly thereafter, such potential estimates have been established extensively to various kinds of equations and systems since these pioneering works (cf. \cite{Lab, LMM, CiSc, ZZM} and the references therein). A further propulsion on this subject was achieved by Mingione \cite{Min}, who showed pointwise gradient estimates for solutions to the equations with linear growth operators in terms of the linear Riesz potentials, exactly as it happens for the Poisson equation via representation formulas. Subsequently, considerable literature deduced that such precise pointwise bounds also held for the gradient of solutions to the nonlinear equations and systems via the nonlinear Wolff type potentials (cf. \cite{DuMi, DuMi2, BH, DuMi3, KuMi, KuMi2, KuMi3, XZ} and the references therein). Besides, we mention some recent works in connecting with the nonlinear Wolff type potential estimates to nonlocal equations and nonuniformly elliptic variational problems (cf. \cite{KuMinSi, BM}).

It is natural to investigate whether the so-called nonlinear potential theory is applicable to the stationary and evolutionary Stokes system. To our knowledge the first results regarding to this subject were given in \cite{MZ}. The authors showed that the weak solution pair to the linear Stokes system inherits the exact analogue of
pointwise potential estimates for elliptic systems. Inspired by the above literature, the purpose of this paper is to extend precise pointwise bounds to the generalized nonlinear Stokes system. The main difficulty is not only that the linear dependence of the extra stress tensor on the shear rate has been replaced by a more general nonlinear relation, but also that the relevant structure only depends on the symmetric part of the gradient.

In order to illustrate the main results of this paper, we start by presenting the following definition of weak solution pair to the generalized Stokes system \eqref{model}, which was initiated by \cite{KT, BC}.
\begin{definition}\label{Def-weaksolution}
Let ${\bf F}\in L^{p'}(\Omega\,,\mathbb{R}^{n\times n}_{\rm sym})$\,. Then there exists a unique function
\begin{equation*}
{\bf u}\in W^{1,p}_{0,\operatorname{div}}\left(\Omega,\mathbb{R}^{n}\right):=\left\{{\bf u}\in W^{1,p}_{0}\left(\Omega,\mathbb{R}^{n}\right)|\operatorname{div}{\bf u}=0 \ {\rm in} \ \Omega\right\},
\end{equation*}
which solves \eqref{model} in the distribution sense, i.e.,
\begin{equation*}
 \int_{\Omega}\big\langle \mathcal{A}(x,D{\bf u})\,,D\phi\big\rangle\operatorname{d}\!x
= \int_{\Omega}\big\langle{\bf F}\,,D\phi\big\rangle\operatorname{d}\!x
\end{equation*}
for any divergence free test function $\phi\in
W_{0,\operatorname{div}}^{1,p}(\Omega,\mathbb{R}^{n})$\,. Meanwhile, if ${\bf u}$ is such weak solution and $\pi\in L^{p'}(\Omega)$ stands for an associated pressure of ${\bf u}$, which satisfies
\begin{equation*}
 \int_{\Omega}\big\langle\mathcal{A}(x,D{\bf u})\,,D\varphi\big\rangle-\pi \operatorname{div}{\bf\varphi}\operatorname{d}\!x =
\int_{\Omega}\big\langle{\bf F}\,,D\varphi\big\rangle\operatorname{d}\!x
\end{equation*}
for any test function ${\bf \varphi}\in
W_{0}^{1,p}(\Omega,\mathbb{R}^{n})$, then $\left({\bf u},\pi\right)$ is called a weak solution pair to \eqref{model}\,, where $p'=\frac{p}{p-1}$ is a conjugate exponent of $p$.
\end{definition}

We shall use the nonlinear Wolff potential to obtain the potential estimates. The definition of the classical Wolff potential is stated below.
\begin{definition}
Let $s>1$ and $\alpha\in(0,\frac{n}{s}]$, the truncated Havin-Maz'ya-Wolff potential ${\bf W}_{\alpha,s}^{R}f$ of $f\in L_{\rm loc}^{1}(\Omega)$ is defined by
\begin{equation*}
  {\bf W}_{\alpha,s}^{R}f(x_{0})=\int_{0}^{R}\left(\varrho^{\alpha s}\fint_{B_{\varrho}(x_{0})}| f(x)|\operatorname{d}\!x\right)^{\frac{1}{s-1}}\frac{\operatorname{d}\!\varrho}{\varrho}
\end{equation*}
for any $x_{0}\in\Omega$ and $R>0$ such that $B_{R}(x_{0})\subset\Omega$, where $B_{\varrho}(x_{0})$ denotes an open ball in $\mathbb{R}^{n}$ with center $x_{0}$ and radius $\varrho>0$.
\end{definition}

In addition, to establish the first pointwise gradient estimates, we need more regularity assumptions on the partial map $x\mapsto\mathcal{A}(x,\cdot)$\,.
\begin{definition}\label{Def-Vanish}
For some $R>0$, we denote the $\operatorname{BMO}$ semi-norm of $x\mapsto\mathcal{A}(x,\cdot)$ as follows
\begin{equation*}\label{smallBMO}
\left[\mathcal{A}\right]_{\operatorname{BMO}}(R):=\sup_{{\substack{ y\,\in\,\Omega\\0<r\leq R}}} \fint_{B_{r}(y)}
\beta(\mathcal{A},B_{r}(y))\operatorname{d}\!x\,,
\end{equation*}
where
\begin{equation*}
\beta(\mathcal{A},B_{r}(y)):=\sup_{{\substack{\xi\in\mathbb{R}^{n\times n}_{sym}\backslash \{\bf 0\}}}}\frac{\left|\mathcal{A}(x,\xi)-\left(\mathcal{A}(\cdot,\xi)\right)_{B_{r}(y)}\right|}{\left(\mu^{2}+|\xi|^{2}\right)^{\frac{p-2}{2}}|\xi|}
\end{equation*}
and
\begin{equation*}
\left(\mathcal{A}(\cdot,\xi)\right)_{B_{r}(y)}:=\fint_{B_{r}(y)}\mathcal{A}(x,\xi)\operatorname{d}\!x
=\frac{1}{\left|B_{r}(y)\right|}\int_{B_{r}(y)}\mathcal{A}(x,\xi)\operatorname{d}\!x\,.
\end{equation*}
We say that $\left[\mathcal{A}\right]_{\operatorname{BMO}}(\cdot)$ is Dini-$\operatorname{BMO}$ regular if
\begin{equation}\label{dini-bmo}
  d(R):=\int_{0}^{R}\left[\mathcal{A}\right]^{\frac{\hat{\sigma}}{p}}_{\operatorname{BMO}}(\varrho)
  \frac{\operatorname{d}\!\varrho}{\varrho}<\infty\,,
\end{equation}
where $\hat{\sigma}=\hat{\sigma}(\nu,L,n,p)>p$ will be given in \eqref{sigma}.
\end{definition}

We are now in a position to state the first result of this paper. It infers that the shear rate $D{\bf u}$ and its pressure $\pi$ to the nonlinear Stokes system \eqref{model} with Dini-$\operatorname{BMO}$ regular coefficients can be controlled by an unconventional nonlinear Havin-Maz'ya-Wolff type potential of the nonhomogeneous term $\bf F$.

\begin{theorem}\label{Th1}\textup{(\textbf{Gradient estimate.})}
Let $({\bf u},\pi)$ be a weak solution pair to
\eqref{model} with ${\bf F}\in L_{\rm loc}
^{p'}\left(\Omega\,,\mathbb{R}^{2\times 2}_{\rm sym}\right)$ and $\mathcal{A}$ satisfying \eqref{structural-cons} and \eqref{dini-bmo} for some $\delta=\delta(p,\nu,L)>0$ and $R>0$. Then there exist a positive constant $C=C(\nu,L,p,\left[\mathcal{A}\right]_{\operatorname{BMO}}(\cdot))$ and a radius $R_{0}=R_{0}\left(\nu,L,p,d(\cdot)\right)$ such that the following pointwise estimate
\begin{eqnarray}\label{gradient-estimate}
&& |D{\bf u}(x_{0})|+|\pi (x_{0})|^{\frac{p'}{p}} \\
&\leq& C\fint_{B_{R}(x_{0})}\left(\mu+\left|D{\bf u}\right|\right)\operatorname{d}\!x+ C\left(\fint_{B_{R}(x_{0})}\left|\pi\right|\operatorname{d}\!x\right)^{\frac{p'}{p}}+C\int_{0}^{2R}
\left(\fint_{B_{\varrho}(x_{0})}|{\bf F}-({\bf F})_{B_{\rho}(x_{0})}|^{p'}\operatorname{d}\!x\right)^{\frac{1}{p}}\frac{\operatorname{d}\!\varrho}{\varrho} \nonumber
\end{eqnarray}
holds for almost all $x_{0}\in\Omega$ and every $B_{2R}(x_{0})\subset\Omega$ with $R\leq R_{0}$.
\end{theorem}

\begin{remark}
The reason for the pointwise estimate \eqref{gradient-estimate} applied only to planar flows is due to the absence of Lipschitz regularity for solutions to the corresponding limiting problem \eqref{ComparisonSystem-2} in higher dimensions. Note that whether such Lipschitz regularity would hold in higher dimensions is still an open problem so far. Even so, such pointwise estimate of symmetric gradient is still new for the case of elliptic systems.
\end{remark}
\begin{remark}
As a consequence of Theorem \ref{Th1}\,, one can find a sufficient condition on the nonhomogeneous term $\bf F$ such that $D{\bf u}$ and $\pi$ are locally bounded in the plane, which solves an open issue in the Calder\'{o}n-Zygmund theory to the nonlinear Stocks system. For instance, if the given exterior force ${\bf F}$ has a modulus of continuity $\omega$ satisfying the Dini type continuous condition that $\int_{0}^{R}\omega^{\frac{p'}{p}}(\varrho)\frac{\operatorname{d}\!\varrho}{\varrho} <\infty$, then the shear rate $D{\bf u}$ and  its pressure $\pi$ are locally bounded in $\Omega$.
\end{remark}
Furthermore, we establish a precise potential estimate for the weak solution $\bf u$ to \eqref{model} without the restriction to the planar case. We should mention that there is no additional regularity assumption on the partial map $x\mapsto\mathcal{A}(x,\cdot)$ in the following theorem.
\begin{theorem}\label{Th2}\textup{(\textbf{Zero order estimate.})}
Let $({\bf u},\pi)$ be a weak solution pair to
\eqref{model} with ${\bf F}\in L
^{p'}_{\rm loc}\left(\Omega\,,\mathbb{R}^{n\times n}_{\rm sym}\right)$ for $2\leq p\leq n<\theta$ and $\mathcal{A}$ satisfying \eqref{structural-cons}, where $\theta$ is given in \eqref{v-reverseholder}\,. Then there exists a positive constant $C=C(n,\nu,L,p)$ such that the pointwise estimate
\begin{equation}\label{zero-estimate}
  |{\bf u}(x_{0})|\leq C\fint_{B_{R}(x_{0})}\left|{\bf u}\right|\operatorname{d}\!x+C\, {\bf W}_{\frac{p}{p+1},p+1}^{2R}\left(\mu^{p}+|{\bf F}|^{p'}\right)(x_{0})
\end{equation}
holds for almost all $x_{0}\in\Omega$ and every $B_{2R}(x_{0})\subset\Omega$.
\end{theorem}

\begin{remark}
Indeed, even for simpler homogeneous elliptic systems of the type $-\operatorname{div}A\left(\nabla{\bf u}\right)=0$, the vector valued solutions $\bf u$ may be unbounded (cf. \cite{MS}).
It means that potential estimate such as \eqref{zero-estimate} does not hold
to general elliptic systems, unless additional assumptions are made on the dimension $n$ or the vector field $A$. The new result of Theorem \ref{Th2} indicates that the potential estimate \eqref{zero-estimate} is valid for $2\leq p\leq n<\theta$, which would imply the local boundedness of weak solution to nonlinear Stokes systems \eqref{model} in this case.
\end{remark}

The remainder of this paper is organized as follows. Section \ref{section2} consists of preliminary results and priori estimates, which will be fundamental to the proof of our main results. Section \ref{section3} is devoted to
the comparison estimates between the localized problem and the associated homogeneous problems. In the last section, we complete the proof of Theorem \ref{Th1} and \ref{Th2}\,, respectively.

\section{Preliminaries.}\label{section2}

In this section, we first provide a number of auxiliary results which are essential for the proof of the main results. And then, we introduce the comparison systems and establish some useful priori estimates.
In what follows, $C$ denotes a constant whose value may be different from line to line, and only the relevant dependence is specified.
\subsection{Auxiliary results.}
Let us begin with the following known technique lemma which is established in \cite[Theorem 4.1]{ADM}.
\begin{lemma}\label{existence-Johndomain}
Let $\Omega\in\mathbb{R}^{n}$, $n\geq2$ be a bounded John domain. Given $g\in L^{p}(\Omega)$ with $1<p<+\infty$ and  $\int_{\Omega}g\operatorname{d}\!x=0$, there exists at least one ${\bf \psi}\in W_{0}^{1,p}\left(\Omega,\mathbb{R}^{n}\right)$ satisfying
\begin{eqnarray*}
  \operatorname{div}{\bf \psi}&=& g \quad\mbox{in}\ \ \Omega\,,\\
  \|{\bf\nabla\psi}\|_{L^{p}(\Omega)} &\leq& C \|g\|_{L^{p}(\Omega)}\,,
\end{eqnarray*}
where the positive constant $C=C\left(\operatorname{diam}(\Omega),n,p\right)$. Particularly, if $\Omega=B_{R}(x_{0})$, then $C$ depends only on $n$ and $p$.
\end{lemma}

In the sequel, we will use the following self-improving property of reverse H\"{o}lder type inequality.
\begin{lemma}\label{reverse-holder} {\rm (cf. \cite[Corollary 3.4]{DKS}.)}
Let ${\bf g},\,{\bf h}\in L_{\rm loc}^{1}(\Omega,\mathbb{R}^{n\times n})$. Assume that there exist constants $0<\tau<1$, $\gamma>1$ and $C_{0}>0$ such that
\begin{equation*}
  \left(\fint_{B_{\tau r}}|{\bf g}|^{\gamma}\operatorname{d}\!x\right)^{\frac{1}{\gamma}}\leq C_{0}\fint_{B_{r}}|{\bf g}|\operatorname{d}\!x+C_{0}\fint_{B_{r}}|{\bf h}|\operatorname{d}\!x
\end{equation*}
for every $B_{r}\subset\Omega$. Then for every $0<t\leq1$, there exists a positive constant $C=C(C_{0},n,\tau,t)$ such that
\begin{equation*}
  \left(\fint_{B_{\tau r}}|{\bf g}|^{\gamma}\operatorname{d}\!x\right)^{\frac{1}{\gamma}}\leq C\left(\fint_{B_{r}}|{\bf g}|^{t}\operatorname{d}\!x\right)^{\frac{1}{t}}+C\fint_{B_{r}}|{\bf h}|\operatorname{d}\!x\,.
\end{equation*}
\end{lemma}
And then, combining the aforementioned Lemma \ref{reverse-holder} with Lemma 3.4 in \cite{DiKaSc}, one obtains a reserve H\"{o}lder type estimate for the symmetric gradient $D\bf u$ to \eqref{model}\,.
\begin{lemma}\label{reserve_Holder_Du}
Let ${\bf u}$ be a local weak solution of \eqref{model} with ${\bf F}\in L_{\rm loc}^{p'}(\Omega\,,\mathbb{R}_{\rm sym}^{n\times n})$\,. Then there exists a positive constant $C$ depending only on $n$, $\nu$, $L$ and $p$, such that
\begin{eqnarray}\label{reholderDu}
  \fint_{B_{R}(x_{0})}\left|D{\bf u}\right|^{p}\operatorname{d}\!x
\leq C\left(\fint_{B_{2R}(x_{0})}\!\left(\mu+\left|D{\bf u}\right|\right)\operatorname{d}\!x\right)^{p}+C\fint_{B_{2R}(x_{0})}\!\left|{\bf F}-{\bf F}_{0}\right|^{p'}\operatorname{d}\!x
\end{eqnarray}
for almost all $x_{0}\in\Omega$, any constant matrix ${\bf F}_{0}\in \mathbb{R}_{\rm sym}^{n\times n}$ and every $B_{2R}(x_{0})\subset\Omega$.
\end{lemma}

Moreover, the following Korn's inequality will play an important role in the analysis of generalized nonlinear Stokes system.
\begin{lemma}\label{Korn} {\rm (cf. \cite{DRS}.)}
Let $B\subset\mathbb{R}^{n}$ be a ball and $1<p<+\infty$. Then for all ${\bf u}\in W_{0}^{1,p}(B, \mathbb{R}^{n})$, there holds that
\begin{equation}\label{Korn1}
  \int_{B}|\nabla{\bf u}|^{p}\operatorname{d}\!x\leq C\int_{B}|D{\bf u}|^{p}\operatorname{d}\!x\,.
\end{equation}
While if ${\bf u}\in W^{1,p}(B, \mathbb{R}^{n})$, then
\begin{equation}\label{Korn2}
  \int_{B}|\nabla{\bf u}-\left(\nabla{\bf u}\right)_{B}|^{p}\operatorname{d}\!x\leq C\int_{B}|D{\bf u}-\left(D{\bf u}\right)_{B}|^{p}\operatorname{d}\!x\,,
\end{equation}
where the positive constant $C$ depends only on $p$.
\end{lemma}

Finally, we will frequently use the following basic estimate to discuss oscillation estimates. Let $E$ be a measurable subset in $\mathbb{R}^{n}$. For any ${\bf f}\in L^{p}(E, \mathbb{R}^{ m})$ with $p\in[1,\infty)$ and $m\geq1$, we have
\begin{equation}\label{eqn-minimal2}
\left(\fint_{E}\left|\,{\bf f}(x)-\left({\bf f}\right)_{E}\,\right|^{p}\operatorname{d}\!x\right)^{\frac{1}{p}}\,\leq\,2\min_{{\bf h}\in\mathbb{R}^{ m}}\left(\fint_{E}\left|\,{\bf f}(x)-{\bf h}\,\right|^{p}\operatorname{d}\!x\right)^{\frac{1}{p}}\,.
\end{equation}

\subsection{Comparison systems and priori estimates.}
In this subsection, we establish some useful priori estimates to the Stokes system \eqref{model}
in comparison with the homogenous problem
\begin{equation}\label{ComparisonSystem}
\left\{\begin{array}{r@{\ \ }c@{\ \ }ll}
\operatorname{div}\mathcal{A}\left(x,D{\bf v}\right)-\bf\nabla\pi_{ v} & =& {\bf 0} & \mbox{in}\ \ B_{2R}(x_{0})\,, \\[0.05cm]
\operatorname{div}{\bf v}&=& 0 & \mbox{in}\ \ B_{2R}(x_{0})\,, \\[0.05cm]
{\bf v}&=&{\bf u} & \mbox{on}\ \ \partial B_{2R}(x_{0})\,,
\end{array}\right.
\end{equation}
and the limiting problem
\begin{equation}\label{ComparisonSystem-2}
\left\{\begin{array}{r@{\ \ }c@{\ \ }ll}
\operatorname{div}\bar{\mathcal{A}}(D{\bf w})-\bf\nabla\pi_{ w} & =&  {\bf 0} & \mbox{in}\ \ B_{\frac{3R}{2}}(x_{0})\,, \\[0.05cm]
\operatorname{div}{\bf w}&=& 0 & \mbox{in}\ \ B_{\frac{3R}{2}}(x_{0})\,, \\[0.05cm]
{\bf w} &=&  {\bf v} & \mbox{on}\ \ \partial B_{\frac{3R}{2}}(x_{0})\,,
\end{array}\right.
\end{equation}
where $\bar{\mathcal{A}}(D{\bf w})=\left(\mathcal{A}\left(\cdot,D{\bf w}\right)\right)_{B_{\frac{3R}{2}}(x_{0})}$, $B_{2R}(x_{0})\subset\Omega$, $\pi_{\bf v}$ and $\pi_{\bf w}$ are the associated pressure terms to $\bf v$ and $\bf w$ respectively.

We first show that the shear rate $D\bf v$ to the homogenous Stokes system \eqref{ComparisonSystem} can be controlled by the shear rate $D\bf u$.
\begin{lemma}\label{Du-control-Dv}
Let $\left({\bf u},\pi\right)$ be a weak solution pair to \eqref{model} with ${\bf F}\in L_{\rm loc}^{p'}(\Omega\,,\mathbb{R}_{\rm sym}^{n\times n})$\,. Then one can find a weak solution pair $\left({\bf v},\pi_{\bf v}\right)$ to \eqref{ComparisonSystem} such that
\begin{eqnarray}\label{DuconDv}
  \fint_{B_{2R}(x_{0})}\left|D{\bf v}\right|^{p}\operatorname{d}\!x
\leq C\fint_{B_{2R}(x_{0})}\!\left(\mu^{p}+\left|D{\bf u}\right|^{p}\right)\operatorname{d}\!x
\end{eqnarray}
for almost all $x_{0}\in\Omega$ and every $B_{2R}(x_{0})\subset\Omega$, where the positive constant $C=C(\nu, L, p)$.
\end{lemma}
\begin{proof}
Since ${\bf u}-{\bf v}\in W_{0,\operatorname{div}}^{1,p}(B_{2R}(x_{0}),\mathbb{R}^{n})$, we choose $ {\bf u}-{\bf v}$ as a divergence free test function to the problem \eqref{ComparisonSystem}. Then it follows that
\begin{equation*}
  \fint_{B_{2R}(x_{0})} \left\langle\,\mathcal{A}\left(x,D{\bf v}\right),D{\bf u}-D{\bf  v}\right\rangle\operatorname{d}\!x=0\,.
\end{equation*}
By virtue of \eqref{structural-cons} and Young's inequality, we derive
\begin{eqnarray}\label{uconv-1}
 && \nu\,\fint_{B_{2R}(x_{0})}\left(\mu^{2}+|D{\bf v}|^{2}\right)^{\frac{p-2}{2}}\left|D{\bf  v}\right|^{2}\operatorname{d}\!x\nonumber\\
&\leq& \,\fint_{B_{2R}(x_{0})} \left\langle\,\mathcal{A}\left(x,D{\bf v}\right),D{\bf  v}\right\rangle\operatorname{d}\!x \nonumber\\
&=& \,\fint_{B_{2R}(x_{0})} \left\langle\,\mathcal{A}\left(x,D{\bf v}\right),D{\bf  u}\right\rangle\operatorname{d}\!x \nonumber\\
&\leq&\epsilon_{1}\,\fint_{B_{2R}(x_{0})}\left|\mathcal{A}\left(x,D{\bf v}\right)\right|^{p'}\operatorname{d}\!x+
C\left(\epsilon_{1},p\right)\fint_{B_{2R}(x_{0})}\left|D{\bf  u}\right|^{p}\operatorname{d}\!x \nonumber\\
&\leq& \epsilon_{1}c_{1}(p,L)\,\fint_{B_{2R}(x_{0})}\left(\mu^{p}+\left|D{\bf v}\right|^{p}\right)\operatorname{d}\!x+
C\left(\epsilon_{1},p\right)\fint_{B_{2R}(x_{0})}\left|D{\bf  u}\right|^{p}\operatorname{d}\!x\,.
\end{eqnarray}
For the case $p\geq2$, one has
\begin{eqnarray*}
&& \nu\,\fint_{B_{2R}(x_{0})}|D{\bf v}|^{p}\operatorname{d}\!x\\
 &\leq& \nu\,\fint_{B_{2R}(x_{0})}\left(\mu^{2}+|D{\bf v}|^{2}\right)^{\frac{p-2}{2}}\left|D{\bf  v}\right|^{2}\operatorname{d}\!x\\
&\leq& \epsilon_{1}c_{1}\,\fint_{B_{2R}(x_{0})}\left(\mu^{p}+\left|D{\bf v}\right|^{p}\right)\operatorname{d}\!x+
C\left(\epsilon_{1},p\right)\fint_{B_{2R}(x_{0})}\left|D{\bf  u}\right|^{p}\operatorname{d}\!x\,.
\end{eqnarray*}
By selecting $\epsilon_{1}=\frac{\nu}{2c_{1}}$\,, we obtain that
\begin{eqnarray}\label{uconv-2}
\fint_{B_{2R}(x_{0})}|D{\bf v}|^{p}\operatorname{d}\!x\leq
C\left(p,\nu,L\right)\fint_{B_{2R}(x_{0})}\left(\mu^{p}+\left|D{\bf  u}\right|^{p}\right)\operatorname{d}\!x\,.
\end{eqnarray}
While for $1<p<2$, we apply Young's inequality and \eqref{uconv-1} to derive
\begin{eqnarray*}
&&\fint_{B_{2R}(x_{0})}|D{\bf v}|^{p}\operatorname{d}\!x\\
&\leq&\epsilon_{2}\fint_{B_{2R}(x_{0})}\left(\mu^{2}+|D{\bf v}|^{2}\right)^{\frac{p}{2}}\operatorname{d}\!x+
 C(\epsilon_{2},p)\fint_{B_{2R}(x_{0})}\left(\mu^{2}+|D{\bf v}|^{2}\right)^{\frac{p-2}{2}}\left|D{\bf  v}\right|^{2}\operatorname{d}\!x\\
&\leq&\epsilon_{2}c_{2}(p)\fint_{B_{2R}(x_{0})}\left(\mu^{p}+|D{\bf v}|^{p}\right)\operatorname{d}\!x +\epsilon_{1}c_{3}(\epsilon_{2},p,L,\nu)\,\fint_{B_{2R}(x_{0})}\left(\mu^{p}+\left|D{\bf v}\right|^{p}\right)\operatorname{d}\!x\\
&&+\,C\left(\epsilon_{1},\epsilon_{2},p,\nu\right)\fint_{B_{2R}(x_{0})}\left|D{\bf  u}\right|^{p}\operatorname{d}\!x\,.
\end{eqnarray*}
Now selecting $\epsilon_{2}=\frac{1}{4c_{2}}$, and then choosing $\epsilon_{1}=\frac{1}{4c_{3}}$, we obtain
\begin{eqnarray}\label{uconv-3}
\fint_{B_{2R}(x_{0})}|D{\bf v}|^{p}\operatorname{d}\!x\leq
C\left(p,\nu,L\right)\fint_{B_{2R}(x_{0})}\left(\mu^{p}+\left|D{\bf  u}\right|^{p}\right)\operatorname{d}\!x\,.
\end{eqnarray}
Finally, a combination of \eqref{uconv-2} and \eqref{uconv-3} yields \eqref{DuconDv} holds for $1<p<+\infty$, which completes the proof of Lemma \ref{Du-control-Dv}\,.
\end{proof}

Next, we deduce that the shear rate $D\bf w$ to the limiting Stokes system \eqref{ComparisonSystem-2} can be estimated in terms of the shear rate $D\bf v$ to the homogeneous Stokes system \eqref{ComparisonSystem}\,.
\begin{lemma}\label{Dv-control-Dw}
Let $\left({\bf v},\pi_{v}\right)$ be a weak solution pair to \eqref{ComparisonSystem}\,. Then one can find a weak solution pair $\left({\bf w},\pi_{\bf w}\right)$ to \eqref{ComparisonSystem-2} such that
\begin{eqnarray}\label{DvconDw}
  \fint_{B_{\frac{3R}{2}}(x_{0})}\left|D{\bf w}\right|^{p}\operatorname{d}\!x
\leq C\fint_{B_{\frac{3R}{2}}(x_{0})}\!\left(\mu^{p}+\left|D{\bf v}\right|^{p}\right)\operatorname{d}\!x
\end{eqnarray}
for almost all $x_{0}\in\Omega$ and every $B_{\frac{3R}{2}}(x_{0})\subset\Omega$, where the positive constant $C=C(\nu, L, p)$.
\end{lemma}
\begin{proof}
Since ${\bf v}-{\bf w}\in W_{0,\operatorname{div}}^{1,p}(B_{\frac{3R}{2}}(x_{0}),\mathbb{R}^{n})$, we select $ {\bf v}-{\bf w}$ as
a divergence free test function for \eqref{ComparisonSystem-2}, that is to say,
\begin{equation*}
  \fint_{B_{\frac{3R}{2}}(x_{0})} \left\langle\,\bar{\mathcal{A}}\left(D{\bf w}\right),D{\bf v}-D{\bf w}\right\rangle\operatorname{d}\!x=0\,.
\end{equation*}
Applying \eqref{structural-cons}, Young's inequality and H\"{o}lder's inequality, one derives that
\begin{eqnarray}\label{vconw-1}
 && \nu\,\fint_{B_{\frac{3R}{2}}(x_{0})}\left(\mu^{2}+|D{\bf w}|^{2}\right)^{\frac{p-2}{2}}\left|D{\bf  w}\right|^{2}\operatorname{d}\!x\nonumber\\
&\leq& \fint_{B_{\frac{3R}{2}}(x_{0})} \left\langle\,\left(\mathcal{A}\left(\cdot,D{\bf w}\right)\right)_{B_{\frac{3R}{2}}(x_{0})},D{\bf w}\right\rangle\operatorname{d}\!x \nonumber\\
&=& \fint_{B_{\frac{3R}{2}}(x_{0})} \left\langle\,\left(\mathcal{A}\left(\cdot,D{\bf w}\right)\right)_{B_{\frac{3R}{2}}(x_{0})},D{\bf v}\right\rangle\operatorname{d}\!x \nonumber\\
&\leq&\epsilon_{3}\,\fint_{B_{\frac{3R}{2}}(x_{0})}\left|\left(\mathcal{A}\left(\cdot,D{\bf w}\right)\right)_{B_{\frac{3R}{2}}(x_{0})}\right|^{p'}\operatorname{d}\!x+
C\left(\epsilon_{3},p\right)\fint_{B_{\frac{3R}{2}}(x_{0})}\left|D{\bf v}\right|^{p}\operatorname{d}\!x \nonumber\\
&\leq& \epsilon_{3}c_{4}(p,L)\,\fint_{B_{\frac{3R}{2}}(x_{0})}\left(\mu^{p}+\left|D{\bf w}\right|^{p}\right)\operatorname{d}\!x+
C\left(\epsilon_{3},p\right)\fint_{B_{\frac{3R}{2}}(x_{0})}\left|D{\bf v}\right|^{p}\operatorname{d}\!x\,.
\end{eqnarray}
By proceeding similarly as in the proof of Lemma \ref{Du-control-Dv}, splitting into two different cases $p\geq2$ and $1<p<2$, and then choosing the appropriate positive constant $\epsilon_{3}$, one concludes that
\begin{eqnarray}\label{vconw-2}
\fint_{B_{\frac{3R}{2}}(x_{0})}|D{\bf w}|^{p}\operatorname{d}\!x\leq
C\left(p,\nu,L\right)\fint_{B_{\frac{3R}{2}}(x_{0})}\left(\mu^{p}+\left|D{\bf  v}\right|^{p}\right)\operatorname{d}\!x
\end{eqnarray}
for any $1<p<+\infty$. Thus the proof of Lemma \ref{Dv-control-Dw} is completed.
\end{proof}

Another basic tool we use is a Caccioppoli type inequality to the Stokes system \eqref{model}, which takes up the rest of this subsection.
\begin{lemma}\label{caccioppoli-u}
Let $\left({\bf u},\pi\right)$ be a weak solution pair to \eqref{model} with ${\bf F}\in L_{\rm loc}^{p'}(\Omega\,,\mathbb{R}_{\rm sym}^{n\times n})$\,. There holds
\begin{equation}\label{caccioppoli}
   \fint_{B_{R}(x_{0})}|D{\bf u}|^{p}\operatorname{d}\!x\leq\frac{C}{R^{p}}\fint_{B_{2R}(x_{0})}|{\bf u}-\left(\bf u\right)_{B_{2R}(x_{0})}|^{p}\operatorname{d}\!x+C\fint_{B_{2R}(x_{0})}\left(\mu^{p}+|{\bf F}-{\bf F}_{0}|^{p'}\right)\operatorname{d}\!x
\end{equation}
for almost all $x_{0}\in\Omega$, any constant matrix ${\bf F}_{0}\in \mathbb{R}_{\rm sym}^{n\times n}$ and every $B_{2R}(x_{0})\subset\Omega$, where the positive constant $C$ depends only on $n$, $\nu$, $L$ and $p$.
\end{lemma}
\begin{proof}
Let $\varphi=\eta^{p}\left({\bf u}-\left(\bf u\right)_{B_{2R}(x_{0})}\right)$ with $\eta\in C_{0}^{\infty}\left(B_{R_{2}}(x_{0})\right)$ be a cut-off function between $0$ and $1$ such that $\eta\equiv1$ in $B_{R_{1}}(x_{0})$, $\eta\equiv0$ in $B_{R_{2}}^{c}(x_{0})$ and $\left|\nabla\eta\right|\leq\frac{C}{R_{2}-R_{1}}$ in $B_{R_{2}}(x_{0})\setminus \overline{B_{R_{1}}(x_{0})}$ for all $R<R_{1}<R_{2}<2R$. We first correct $\varphi$ to be a divergence free function by
virtue of the Bogovski\u{\i} operator "Bog" which is introduced in \cite{Bog}. Let $\psi=\operatorname{Bog}\left(\operatorname{div}\varphi\right)$ be a special solution to the following auxiliary problem
\begin{equation}\label{auxiliarypro}
\left\{\begin{array}{r@{\ \ }c@{\ \ }ll}
\operatorname{div}\psi & =& \operatorname{div}\varphi& \mbox{in}\ \ B_{R_{2}}(x_{0})\,, \\[0.05cm]
\psi&=& 0 & \mbox{on}\ \ \partial B_{R_{2}}(x_{0})\,.
\end{array}\right.
\end{equation}
As a consequence of \cite[Theorem 6.6]{DRS}, we have
\begin{equation}\label{r-1}
  \fint_{B_{R_{2}}(x_{0})}|D\psi|^{p}\operatorname{d}\!x\leq C(n,p)\fint_{B_{R_{2}}(x_{0})}|\operatorname{div}\varphi|^{p}\operatorname{d}\!x\,.
\end{equation}
Then it is clear to select $\phi=\varphi-\psi\in
W_{0,\operatorname{div}}^{1,p}(B_{R_{2}}(x_{0}),\mathbb{R}^{n})$ as a divergence free text function to \eqref{model}, that is to say
\begin{equation*}
  \int_{B_{R_{2}}(x_{0})}\big\langle \mathcal{A}(x,D{\bf u})\,,D\phi\big\rangle\operatorname{d}\!x
= \int_{B_{R_{2}}(x_{0})}\big\langle{\bf F}-{\bf F}_{0}\,,D\phi\big\rangle\operatorname{d}\!x,
\end{equation*}
where
$$D\phi=\eta^{p}D{\bf u}+p\eta^{p-1}\frac{\left({\bf u}-\left(\bf u\right)_{B_{2R}(x_{0})}\right)\otimes\nabla\eta+\left(\left({\bf u}-\left(\bf u\right)_{B_{2R}(x_{0})}\right)\otimes\nabla\eta\right)^{T}}{2}-D\psi.$$
By a direct computation together with \eqref{structural-cons}, we have
\begin{eqnarray*}
   &&\nu \fint_{B_{R_{2}}(x_{0})}\eta^{p}\left(\mu^{2}+|D{\bf u}|^{2}\right)^{\frac{p-2}{2}}|D{\bf u}|^{2}\operatorname{d}\!x \\
   &\leq& Lp\fint_{B_{R_{2}}(x_{0})}\eta^{p-1}\left(\mu^{2}+|D{\bf u}|^{2}\right)^{\frac{p-2}{2}}|D{\bf u}|\left|\left({\bf u}-\left(\bf u\right)_{B_{2R}(x_{0})}\right)\otimes\nabla\eta\right|\operatorname{d}\!x\\
   &&+\,L\fint_{B_{R_{2}}(x_{0})}
   \left(\mu^{2}+|D{\bf u}|^{2}\right)^{\frac{p-2}{2}}|D{\bf u}||D\psi|\operatorname{d}\!x \\
   &&+\, \fint_{B_{R_{2}}(x_{0})}|{\bf F}-{\bf F}_{0}|\left(\eta^{p}|D{\bf u}|+p\eta^{p-1}\left|\left({\bf u}-\left(\bf u\right)_{B_{2R}(x_{0})}\right)\otimes\nabla\eta\right|+|D\psi|\right)\operatorname{d}\!x\,.
\end{eqnarray*}
Next by applying Young's inequality, \eqref{r-1} and the assumptions of $\eta$ to the above inequalities, and combining with $\operatorname{div}u=0$, we deduce that
\begin{eqnarray*}
  && \fint_{B_{R_{2}}(x_{0})}\eta^{p}|D{\bf u}|^{p}\operatorname{d}\!x \\
  &\leq&\epsilon_{4}\fint_{B_{R_{2}}(x_{0})}\eta^{p}|D{\bf u}|^{p}\operatorname{d}\!x+C(\epsilon_{4},p) \fint_{B_{R_{2}}(x_{0})}\eta^{p}\left(\mu^{2}+|D{\bf u}|^{2}\right)^{\frac{p-2}{2}}|D{\bf u}|^{2}\operatorname{d}\!x\\
   &\leq& \frac{1}{4}\fint_{B_{R_{2}}(x_{0})}|D{\bf u}|^{p}\operatorname{d}\!x+(\epsilon_{4}+\epsilon_{5})\fint_{B_{R_{2}}(x_{0})}\eta^{p}|D{\bf u}|^{p}\operatorname{d}\!x+C(p,\nu, L)\fint_{B_{R_{2}}(x_{0})}|D\psi|^{p}\operatorname{d}\!x \\
   &&+\frac{C(\epsilon_{4},\epsilon_{5},L,p,\nu)}{(R_{2}-R_{1})^{p}}\fint_{B_{R_{2}}(x_{0})}|{\bf u}-\left(\bf u\right)_{B_{2R}(x_{0})}|^{p}\operatorname{d}\!x +\,C(\epsilon_{4},\epsilon_{5},L,p,\nu)\fint_{B_{R_{2}}(x_{0})}\left(\mu^{p}+|{\bf F}-{\bf F}_{0}|^{p'}\right)\operatorname{d}\!x\\
   &\leq& \frac{1}{4}\fint_{B_{R_{2}}(x_{0})}|D{\bf u}|^{p}\operatorname{d}\!x+(\epsilon_{4}+\epsilon_{5})\fint_{B_{R_{2}}(x_{0})}\eta^{p}|D{\bf u}|^{p}\operatorname{d}\!x+C(n,p,\nu, L)\fint_{B_{R_{2}}(x_{0})}|\operatorname{div}\varphi|^{p}\operatorname{d}\!x \\
   &&+\frac{C(\epsilon_{4},\epsilon_{5},L,p,\nu)}{(R_{2}-R_{1})^{p}}\fint_{B_{R_{2}}(x_{0})}|{\bf u}-\left(\bf u\right)_{B_{2R}(x_{0})}|^{p}\operatorname{d}\!x +\,C(\epsilon_{4},\epsilon_{5},L,p,\nu)\fint_{B_{R_{2}}(x_{0})}\left(\mu^{p}+|{\bf F}-{\bf F}_{0}|^{p'}\right)\operatorname{d}\!x\\
   &=& \frac{1}{4}\fint_{B_{R_{2}}(x_{0})}|D{\bf u}|^{p}\operatorname{d}\!x+(\epsilon_{4}+\epsilon_{5})\fint_{B_{R_{2}}(x_{0})}\eta^{p}|D{\bf u}|^{p}\operatorname{d}\!x+C\fint_{B_{R_{2}}(x_{0})}\left|p\eta^{p-1}\nabla\eta\cdot\left({\bf u}-\left(\bf u\right)_{B_{2R}(x_{0})}\right)\right|^{p}\operatorname{d}\!x \\
   &&+\frac{C}{(R_{2}-R_{1})^{p}}\fint_{B_{R_{2}}(x_{0})}|{\bf u}-\left(\bf u\right)_{B_{2R}(x_{0})}|^{p}\operatorname{d}\!x +\,C\fint_{B_{R_{2}}(x_{0})}\left(\mu^{p}+|{\bf F}-{\bf F}_{0}|^{p'}\right)\operatorname{d}\!x\\
   &\leq& \frac{1}{4}\fint_{B_{R_{2}}(x_{0})}|D{\bf u}|^{p}\operatorname{d}\!x+(\epsilon_{4}+\epsilon_{5})\fint_{B_{R_{2}}(x_{0})}\eta^{p}|D{\bf u}|^{p}\operatorname{d}\!x+\frac{C(\epsilon_{4},\epsilon_{5},n,L,p,\nu)}{(R_{2}-R_{1})^{p}}\fint_{B_{R_{2}}(x_{0})}|{\bf u}-\left(\bf u\right)_{B_{2R}(x_{0})}|^{p}\operatorname{d}\!x \\
   && +\,C(\epsilon_{4},\epsilon_{5},L,p,\nu)\fint_{B_{R_{2}}(x_{0})}\left(\mu^{p}+|{\bf F}-{\bf F}_{0}|^{p'}\right)\operatorname{d}\!x.
\end{eqnarray*}
Furthermore, choosing the positive constants $\epsilon_{4}$ and $\epsilon_{5}$ such that $\epsilon_{4}+\epsilon_{5}=\frac{1}{2}$ and combining with $\eta\equiv1$ in $B_{R_{1}}(x_{0})$ to derive
\begin{eqnarray*}
   \fint_{B_{R_{1}}(x_{0})}|D{\bf u}|^{p}\operatorname{d}\!x
   &\leq& \frac{1}{2}\fint_{B_{R_{2}}(x_{0})}|D{\bf u}|^{p}\operatorname{d}\!x+\frac{C(n,L,p,\nu)}{(R_{2}-R_{1})^{p}}\fint_{B_{2R}(x_{0})}|{\bf u}-\left(\bf u\right)_{B_{2R}(x_{0})}|^{p}\operatorname{d}\!x \\
   && +\,C(n,L,p,\nu)\fint_{B_{2R}(x_{0})}\left(\mu^{p}+|{\bf F}-{\bf F}_{0}|^{p'}\right)\operatorname{d}\!x
\end{eqnarray*}
for all $R<R_{1}<R_{2}<R$.

Finally, by virtue of the well-known iteration lemma, which can be found in \cite[Lemma 6.1]{Giu},
then we conclude that the Caccioppoli type inequality \eqref{caccioppoli} holds. This completes the proof of Lemma \ref{caccioppoli-u}\,.
\end{proof}

\section{Comparison estimates.}\label{section3}

This section is devoted to compare the weak solution pair to \eqref{model} to that of limiting problem \eqref{ComparisonSystem-2} for which we have known regularity results.
We start by establishing a comparison estimate regarding to
$D{\bf u}$ with $D{\bf v}$\,, as well as the associated
pressures $\pi$ with $\pi_{\bf v}$\,.
\begin{lemma}\label{approximation1}
Let $\left({\bf u},\pi\right)$ be a weak solution pair to \eqref{model} with ${\bf F}\in L_{\rm loc}^{p'}(\Omega\,,\mathbb{R}_{\rm sym}^{n\times n})$\,. Then there exists a weak solution pair $\left({\bf v},\pi_{\bf v}\right)$ to \eqref{ComparisonSystem} such that
\begin{eqnarray}\label{inequ-comparisonlemma1}
&& \fint_{B_{2R}(x_{0})}\left(\left|D{\bf u}-D{\bf v}\right|^{p}+\left|\pi-\pi_{\bf v}\right|^{p'}\right)\operatorname{d}\!x \\
&\leq& \varepsilon\chi_{\{p\neq2\}}\left(\fint_{B_{4R}(x_{0})}\left(\mu+|D{\bf u}|\right)\operatorname{d}\!x\right)^{p}\!+C\fint_{B_{4R}(x_{0})}\left|{\bf F}-{\bf F}_{0}\right|^{p'}\operatorname{d}\!x \nonumber
\end{eqnarray}
and
\begin{equation}\label{u-v}
   \fint_{B_{2R}(x_{0})}\left|{\bf u}-{\bf v}\right|^{p}\operatorname{d}\!x\leq \varepsilon R^{p}\chi_{\{1<p<2\}}\left(\fint_{B_{4R}(x_{0})}\left(\mu+|D{\bf u}|\right)\operatorname{d}\!x\right)^{p}+ CR^{p}\fint_{B_{4R}(x_{0})}\!\left|{\bf F}-{\bf F}_{0}\right|^{p'}\operatorname{d}\!x
\end{equation}
for any $0<\varepsilon<1$, any constant matrix ${\bf F}_{0}\in \mathbb{R}_{\rm sym}^{n\times n}$, almost all $x_{0}\in\Omega$ and every $B_{4R}(x_{0})\subset\Omega$\,. Here the positive constant $C=C(\varepsilon, n,\nu, L, p)$.
\end{lemma}
\begin{proof}
Let $\left({\bf u},\pi\right)$ and $\left({\bf v},\pi_{\bf v}\right)$ be the weak solution pairs to \eqref{model} and \eqref{ComparisonSystem} respectively. Then $$({\bf u}-{\bf v}, \pi-\pi_{\bf v})\in W_{0,\operatorname{div}}^{1,p}(B_{2R}(x_{0}), \mathbb{R}^{n})\times L^{p'}(B_{2R}(x_{0}))$$ is a weak solution pair to
\begin{equation}\label{lemu-v-equ1}
\left\{\begin{array}{r@{\ \ }c@{\ \ }ll}
\operatorname{div}\left(\mathcal{A}\left(x,D{\bf u}\right)-\mathcal{A}\left(x,D{\bf v}\right)\right)-\nabla\left(\pi-\pi_{\bf v}\right) & =& \operatorname{div}\left({\bf F}-{\bf F}_{0}\right) & \mbox{in}\ \ B_{2R}(x_{0})\,, \\[0.05cm]
\operatorname{div}\left({\bf u}-{\bf v}\right)&=& 0 & \mbox{in}\ \ B_{2R}(x_{0})\,, \\[0.05cm]
{\bf u}-{\bf v}&=& {\bf 0}& \mbox{on}\ \ \partial B_{2R}(x_{0})\,.
\end{array}\right.
\end{equation}
We choose $ {\bf u}-{\bf v}$ as
a divergence free test function for \eqref{lemu-v-equ1}\,. Then by virtue of \eqref{structural-cons} and Young's inequality, we obtain that
\begin{eqnarray}\label{lemu-v-inequ2}
 && \nu\,\fint_{B_{2R}(x_{0})}\left(\mu^{2}+|D{\bf u}|^{2}+|D{\bf v}|^{2}\right)^{\frac{p-2}{2}}\left|D{\bf u}-D{\bf  v}\right|^{2}\operatorname{d}\!x\nonumber\\
&\leq& \,\fint_{B_{2R}(x_{0})} \left\langle\,\mathcal{A}\left(x,D{\bf u}\right)-\mathcal{A}\left(x,D{\bf v}\right),D{\bf u}-D{\bf  v}\right\rangle\operatorname{d}\!x \nonumber\\
&=& \fint_{B_{2R}(x_{0})}\left\langle{\bf F}-{\bf F}_{0}\,,D{\bf u}-D{\bf v}\right\rangle\,\operatorname{d}\!x \nonumber\\
&\leq&\tau_{1}\fint_{B_{2R}(x_{0})}\left|D{\bf  u}-D{\bf v}\right|^{p}\operatorname{d}\!x+ C(\tau_{1},p)\,\fint_{B_{2R}(x_{0})}\left|{\bf F}-{\bf F}_{0}\right|^{p'}\operatorname{d}\!x.
\end{eqnarray}
In order to estimate the first term on the right side of \eqref{lemu-v-inequ2}, we shall consider the following two cases that $p\geq2$ and $1<p<2$. For the former one, we have
\begin{eqnarray}\label{uv-p>=2}
  \left|D{\bf  u}-D{\bf v}\right|^{p} &=& \left|D{\bf  u}-D{\bf v}\right|^{p-2}\left|D{\bf  u}-D{\bf v}\right|^{2} \nonumber \\
   &\leq& \left(\left|D{\bf  u}\right|+\left|D{\bf v}\right|\right)^{p-2}\left|D{\bf  u}-D{\bf v}\right|^{2} \nonumber\\
   &\leq& 2^{\frac{p-2}{2}}\left(\mu^{2}+\left|D{\bf  u}\right|^{2}+\left|D{\bf v}\right|^{2}\right)^{\frac{p-2}{2}}\left|D{\bf  u}-D{\bf v}\right|^{2}\,.
\end{eqnarray}
Inserting \eqref{uv-p>=2} into \eqref{lemu-v-inequ2} and
choosing $\tau_{1}=\frac{\nu}{2^{\frac{p}{2}}}$\,, we derive
\begin{eqnarray*}
 && \fint_{B_{2R}(x_{0})}\left(\mu^{2}+|D{\bf u}|^{2}+|D{\bf v}|^{2}\right)^{\frac{p-2}{2}}\left|D{\bf u}-D{\bf  v}\right|^{2}\operatorname{d}\!x\nonumber\\
 &\leq& C(\nu,p)\,\fint_{B_{2R}(x_{0})}\left|{\bf F}-{\bf F}_{0}\right|^{p'}\operatorname{d}\!x\,.
\end{eqnarray*}
We apply \eqref{uv-p>=2} again to deduce that
\begin{equation}\label{lemuv-p>=2}
   \fint_{B_{2R}(x_{0})}\left|D{\bf  u}-D{\bf v}\right|^{p}\operatorname{d}\!x\leq  C(\nu,p)\,\fint_{B_{2R}(x_{0})}\left|{\bf F}-{\bf F}_{0}\right|^{p'}\operatorname{d}\!x\,.
\end{equation}
While if $1<p<2$, then it follows from Young's inequality that
\begin{eqnarray}\label{uv-1<p<2}
  \left|D{\bf  u}-D{\bf v}\right|^{p} &=& \left[\left(\mu^{2}+|D{\bf u}|^{2}+|D{\bf v}|^{2}\right)^{\frac{p-2}{2}}\left|D{\bf  u}-D{\bf v}\right|^{2}\right]^{\frac{p}{2}}\left(\mu^{2}+|D{\bf u}|^{2}+|D{\bf v}|^{2}\right)^{\frac{p(2-p)}{4}} \nonumber \\
   &\leq& \frac{p}{2}\left(\frac{2\tau_{2}}{2-p}\right)^{-\frac{2-p}{p}}\left(\mu^{2}+\left|D{\bf  u}\right|^{2}+\left|D{\bf v}\right|^{2}\right)^{\frac{p-2}{2}}\left|D{\bf  u}-D{\bf v}\right|^{2}\nonumber\\
   &&+\,\tau_{2}\left(\mu^{2}+\left|D{\bf  u}\right|^{2}+\left|D{\bf v}\right|^{2}\right)^{\frac{p}{2}}\,.
\end{eqnarray}
Combining \eqref{uv-1<p<2} with \eqref{lemu-v-inequ2}, we obtain
\begin{eqnarray*}
&& \fint_{B_{2R}(x_{0})}\left|D{\bf u}-D{\bf  v}\right|^{p}\operatorname{d}\!x\\
  &\leq& \frac{p}{2}\left(\frac{2\tau_{2}}{2-p}\right)^{-\frac{2-p}{p}}\fint_{B_{2R}(x_{0})}\left(\mu^{2}+|D{\bf u}|^{2}+|D{\bf v}|^{2}\right)^{\frac{p-2}{2}}\left|D{\bf u}-D{\bf  v}\right|^{2}\operatorname{d}\!x\\
  &&+\,\tau_{2}\fint_{B_{2R}(x_{0})}\left(\mu^{2}+|D{\bf u}|^{2}+|D{\bf v}|^{2}\right)^{\frac{p}{2}}\operatorname{d}\!x \\
&\leq&\frac{p}{2\nu}\left(\frac{2\tau_{2}}{2-p}\right)^{-\frac{2-p}{p}}\tau_{1}\fint_{B_{2R}(x_{0})}\left|D{\bf  u}-D{\bf v}\right|^{p}\operatorname{d}\!x+\tau_{2}\fint_{B_{2R}(x_{0})}\left(\mu^{2}+|D{\bf u}|^{2}+|D{\bf v}|^{2}\right)^{\frac{p}{2}}\operatorname{d}\!x\\
&&+\,C(\tau_{1},\tau_{2},p,\nu)\,\fint_{B_{2R}(x_{0})}\left|{\bf F}-{\bf F}_{0}\right|^{p'}\operatorname{d}\!x\,.
\end{eqnarray*}
Then we select the positive constant $\tau_{1}=\frac{\nu}{p}\left(\frac{2\tau_{2}}{2-p}\right)^{\frac{2-p}{p}}$ and apply Lemma \ref{Du-control-Dv} to derive that
\begin{eqnarray}\label{lemuv-1<p<2}
 &&\fint_{B_{2R}(x_{0})}\left|D{\bf u}-D{\bf  v}\right|^{p}\operatorname{d}\!x\nonumber\\
 &\leq& 2\tau_{2}\fint_{B_{2R}(x_{0})}\left(\mu^{2}+|D{\bf u}|^{2}+|D{\bf v}|^{2}\right)^{\frac{p}{2}}\operatorname{d}\!x+C(\tau_{2},\nu,p)\,\fint_{B_{2R}(x_{0})}\left|{\bf F}-{\bf F}_{0}\right|^{p'}\operatorname{d}\!x \nonumber\\
 &\leq& C(p,\nu,L)\tau_{2}\fint_{B_{2R}(x_{0})}\left(\mu^{p}+|D{\bf u}|^{p}\right)\operatorname{d}\!x+C(\tau_{2},\nu,p)\,\fint_{B_{2R}(x_{0})}\left|{\bf F}-{\bf F}_{0}\right|^{p'}\operatorname{d}\!x\,.
\end{eqnarray}
Thus, combining \eqref{lemuv-p>=2} with \eqref{lemuv-1<p<2} and using Lemma \ref{reserve_Holder_Du}\,, we deduce that
\begin{equation}\label{lemu-v-inequ3}
  \fint_{B_{2R}(x_{0})}\left|D{\bf u}-D{\bf  v}\right|^{p}\operatorname{d}\!x \leq \varepsilon_{1}\chi_{\{1<p<2\}}\left(\fint_{B_{4R}(x_{0})}\left(\mu+|D{\bf u}|\right)\operatorname{d}\!x\right)^{p}+C\,\fint_{B_{4R}(x_{0})}\left|{\bf F}-{\bf F}_{0}\right|^{p'}\operatorname{d}\!x
\end{equation}
for any $0<\varepsilon_{1}<1$ and $C=C\left(\varepsilon_{1},n,\nu,p,L\right)$.

Next, in order to prove \eqref{u-v}, we apply Poincar\'{e}'s inequality, Korn's inequlity \eqref{Korn1} and \eqref{lemu-v-inequ3} to derive that
\begin{eqnarray*}
\fint_{B_{2R}(x_{0})}\left|{\bf u}-{\bf v}\right|^{p}\operatorname{d}\!x&\leq& C(n,p)R^{p}\fint_{B_{2R}(x_{0})}\!\left|\nabla{\bf u}-\nabla{\bf v}\right|^{p}\operatorname{d}\!x \\
&\leq& C(n,p)R^{p}\fint_{B_{2R}(x_{0})}\!\left|D{\bf u}-D{\bf v}\right|^{p}\operatorname{d}\!x \\
   &\leq& \varepsilon_{2}R^{p}\chi_{\{1<p<2\}}\left(\fint_{B_{4R}(x_{0})}\left(\mu+|D{\bf u}|\right)\operatorname{d}\!x\right)^{p}+C\,R^{p}\fint_{B_{4R}(x_{0})}\left|{\bf F}-{\bf F}_{0}\right|^{p'}\operatorname{d}\!x
\end{eqnarray*}
for any $0<\varepsilon_{2}<1$ and $C=C(\varepsilon_{2}, n,\nu,L,p)$\,, which implies that \eqref{u-v} is true.

In the sequel, we establish the comparison estimate between $\pi$ and $\pi_{\bf v}$.
Let $\varphi\in W_{0}^{1,p}\left(B_{2R}(x_{0}),\mathbb{R}^{n}\right)$ be a test function of \eqref{lemu-v-equ1}\,, then
\begin{equation}\label{lemu-v-equ4}
  \int_{B_{2R}(x_{0})}\left(\pi-\pi_{\bf v}\right)\operatorname{div}\varphi \operatorname{d}\!x= \int_{B_{2R}(x_{0})}\big\langle \mathcal{A}\left(x,D{\bf  u}\right)-\mathcal{A}\left(x,D{\bf v}\right)-({\bf F}-{\bf F}_{0}),D\varphi\big\rangle \operatorname{d}\!x\,.
\end{equation}
More precisely, we select the above ${\bf \varphi}$ be a solution to the following auxiliary problem
\begin{equation}\label{auxiliaryprob}
\left\{\begin{array}{r@{\ \ }c@{\ \ }ll}
\operatorname{div}{\bf \varphi} & =& \operatorname{sgn}\left(\pi-\pi_{\bf v} \right)\left|\pi-\pi_{\bf v} \right|^{\frac{1}{p-1}}
-\left(\operatorname{sgn}\left(\pi-\pi_{\bf v} \right)\left|\pi-\pi_{\bf v} \right|^{\frac{1}{p-1}}\right)_{B_{2R}(x_{0})} & \mbox{in}\ \ B_{2R}(x_{0})\,, \\[0.05cm]
{\bf \varphi} &=&  0 & \mbox{on}\ \ \partial B_{2R}(x_{0})\,.
\end{array}\right.
\end{equation}
If we denote
$$g:= \operatorname{sgn}\left(\pi-\pi_{\bf v} \right)\left|\pi-\pi_{\bf v} \right|^{\frac{1}{p-1}}
-\left(\operatorname{sgn}\left(\pi-\pi_{\bf v} \right)\left|\pi-\pi_{\bf v} \right|^{\frac{1}{p-1}}\right)_{B_{2R}(x_{0})},$$
then it is obvious that $g\in L^{p}(B_{2R}(x_{0}))$ and $\int_{B_{2R}(x_{0})}g(x)\operatorname{d}\!x=0$.
The existence of such a solution to auxiliary problem
\eqref{auxiliaryprob} is guaranteed by Lemma \ref{existence-Johndomain}
and hence
\begin{equation}\label{lemu-v-inequ5}
  \fint_{B_{2R}(x_{0})}\left|\nabla\varphi\right|^{p}\operatorname{d}\!x\leq C\fint_{B_{2R}(x_{0})}\left|\pi-\pi_{\bf v}\right|^{p'}\operatorname{d}\!x\,,
\end{equation}
where the positive constant $C$ depends only on $n$ and $p$. Moreover, by virtue of Lemma 2.8 in \cite{BC}, there exists a unique $\pi-\pi_{\bf v}\in L^{p'}(B_{2R}(x_{0}))$ to \eqref{lemu-v-equ1} such that $\int_{B_{2R}(x_{0})}\left(\pi-\pi_{\bf v}\right)\operatorname{d}\!x=0$.

By substituting such $\varphi$ into equality
\eqref{lemu-v-equ4} and combining Young's inequality with \eqref{structural-cons} and
\eqref{lemu-v-inequ5}\,, one computes that
\begin{eqnarray*}
&& \int_{B_{2R}(x_{0})}\left|\pi-\pi_{\bf v}\right|^{p'} \operatorname{d}\!x \nonumber\\
&=&\int_{B_{2R}(x_{0})}\left|\pi-\pi_{\bf v}\right|^{p'} \operatorname{d}\!x-\left(\operatorname{sgn}\left(\pi-\pi_{\bf v} \right)\left|\pi-\pi_{\bf v} \right|^{\frac{1}{p-1}}\right)_{B_{2R}(x_{0})}\int_{B_{2R}(x_{0})}\left(\pi-\pi_{\bf v}\right)\operatorname{d}\!x \\
&=&\int_{B_{2R}(x_{0})}\left(\pi-\pi_{\bf v}\right)\left[\operatorname{sgn}\left(\pi-\pi_{\bf v} \right)\left|\pi-\pi_{\bf v} \right|^{\frac{1}{p-1}}-\left(\operatorname{sgn}\left(\pi-\pi_{\bf v} \right)\left|\pi-\pi_{\bf v} \right|^{\frac{1}{p-1}}\right)_{B_{2R}(x_{0})}\right] \operatorname{d}\!x \nonumber\\
   &=& \int_{B_{2R}(x_{0})}\big\langle \mathcal{A}\left(x,D{\bf  u}\right)-\mathcal{A}\left(x,D{\bf v}\right)-({\bf F}-{\bf F}_{0}),D\varphi\big\rangle \operatorname{d}\!x \nonumber\\
  &\leq& C(\tau_{3},p)\int_{B_{2R}(x_{0})}\left(\left|\mathcal{A}\left(x,D{\bf  u}\right)-\mathcal{A}\left(x,D{\bf v}\right)\right|^{p'}+\left|{\bf F}-{\bf F}_{0}\right|^{p'}\right)\operatorname{d}\!x+\tau_{3} \int_{B_{2R}(x_{0})}\left|D\varphi\right|^{p}\operatorname{d}\!x \nonumber \\
  &\leq& C(\tau_{3},L,p)\left\{\int_{B_{2R}(x_{0})}\!\left[\left(\mu^{2}+\left|D{\bf  u}\right|^{2}+\left|D{\bf  v}\right|^{2}\right)^{\frac{p-2}{2}}|D{\bf u}-D{\bf v}|\right]^{p'}\!\operatorname{d}\!x+\int_{B_{2R}(x_{0})}\!\left|{\bf F}-{\bf F}_{0}\right|^{p'}\operatorname{d}\!x\right\}\\
  && +\,\tau_{3}\,C(p) \int_{B_{2R}(x_{0})}\left|\nabla\varphi\right|^{p}\operatorname{d}\!x \nonumber \\
  &\leq& C(\tau_{3},L,p)\left\{\int_{B_{2R}(x_{0})}\left[\left(\mu^{2}+\left|D{\bf  u}\right|^{2}+\left|D{\bf  v}\right|^{2}\right)^{\frac{p-2}{2}}|D{\bf u}-D{\bf v}|\right]^{p'}\operatorname{d}\!x+\int_{B_{2R}(x_{0})}\left|{\bf F}-{\bf F}_{0}\right|^{p'}\operatorname{d}\!x\right\}\\
  && +\,\tau_{3}\,C_{1}(n,p)\int_{B_{2R}(x_{0})}\left|\pi-\pi_{\bf v}\right|^{p'}\operatorname{d}\!x\,.
\end{eqnarray*}
By choosing the positive constant $\tau_{3}=\frac{1}{2C_{1}}$\,, one infers that
\begin{eqnarray}\label{lemu-v-inequ6}
 && \fint_{B_{2R}(x_{0})}\left|\pi-\pi_{\bf v}\right|^{p'} \operatorname{d}\!x \\
  &\leq&C\fint_{B_{2R}(x_{0})}\left[\left(\mu^{2}+\left|D{\bf  u}\right|^{2}+\left|D{\bf  v}\right|^{2}\right)^{\frac{p-2}{2}}|D{\bf u}-D{\bf v}|\right]^{p'}\operatorname{d}\!x+C\fint_{B_{2R}(x_{0})}\left|{\bf F}-{\bf F}_{0}\right|^{p'}\operatorname{d}\!x\,, \nonumber
\end{eqnarray}
where $C=C(n,p,L)$. The following estimate is also split into two cases according to the value of $p$.
If $1<p\leq2$, then we have
\begin{eqnarray}\label{lemu-v-inequ7}
 && \fint_{B_{2R}(x_{0})}\left|\pi-\pi_{\bf v}\right|^{p'} \operatorname{d}\!x \\
  &\leq&C\fint_{B_{2R}(x_{0})}|D{\bf u}-D{\bf v}|^{p}\operatorname{d}\!x+C\fint_{B_{2R}(x_{0})}\left|{\bf F}-{\bf F}_{0}\right|^{p'}\operatorname{d}\!x \nonumber\\
  &\leq& C(n,p,L)\varepsilon_{1}\chi_{\{1<p<2\}}\left(\fint_{B_{4R}(x_{0})}\left(\mu+|D{\bf u}|\right)\operatorname{d}\!x\right)^{p}+C(\varepsilon_{1},n,p,\nu,L)\,\fint_{B_{4R}(x_{0})}\left|{\bf F}-{\bf F}_{0}\right|^{p'}\operatorname{d}\!x\,, \nonumber
\end{eqnarray}
which is ensured by \eqref{lemu-v-inequ3}. While the situation is however different when $p>2$, we combine Young's inequality with Lemma \ref{Du-control-Dv} and \eqref{lemu-v-inequ3} to derive
\begin{eqnarray}\label{lemu-v-inequ8}
   && \fint_{B_{2R}(x_{0})}\left[\left(\mu^{2}+\left|D{\bf  u}\right|^{2}+\left|D{\bf  v}\right|^{2}\right)^{\frac{p-2}{2}}|D{\bf u}-D{\bf v}|\right]^{p'}\operatorname{d}\!x \nonumber\\
   &\leq& \tau_{4}\fint_{B_{2R}(x_{0})}\left(\mu^{p}+\left|D{\bf  u}\right|^{p}+\left|D{\bf  v}\right|^{p}\right)\operatorname{d}\!x+C(\tau_{4},p)\fint_{B_{2R}(x_{0})}|D{\bf u}-D{\bf v}|^{p}\operatorname{d}\!x \nonumber\\
   &\leq& \tau_{4}\,C(\nu,L,p)\fint_{B_{2R}(x_{0})}\left(\mu^{p}+\left|D{\bf  u}\right|^{p}\right)\operatorname{d}\!x+C(\tau_{4},p,\nu,L)\fint_{B_{2R}(x_{0})}\left|{\bf F}-{\bf F}_{0}\right|^{p'}\operatorname{d}\!x\,.
\end{eqnarray}
By inserting \eqref{lemu-v-inequ8} into \eqref{lemu-v-inequ6} and applying Lemma \ref{reserve_Holder_Du} again, we obtain
\begin{eqnarray}\label{lemu-v-inequ9}
&& \fint_{B_{2R}(x_{0})}\left|\pi-\pi_{\bf v}\right|^{p'} \operatorname{d}\!x\nonumber\\
  &\leq& \tau_{4}\,C(n,\nu,L,p)\left(\fint_{B_{4R}(x_{0})}\left(\mu+\left|D{\bf  u}\right|\right)\operatorname{d}\!x\right)^{p}+C(\tau_{4},n,p,\nu,L)\fint_{B_{4R}(x_{0})}\left|{\bf F}-{\bf F}_{0}\right|^{p'}\operatorname{d}\!x\,.
\end{eqnarray}
Thus, a combination of \eqref{lemu-v-inequ7} and \eqref{lemu-v-inequ9} yields that
\begin{equation}\label{lempi-pi_v}
  \fint_{B_{2R}(x_{0})}\left|\pi-\pi_{\bf v}\right|^{p'} \operatorname{d}\!x\leq\varepsilon_{3}\chi_{\{p\neq2\}}\left(\fint_{B_{4R}(x_{0})}\left(\mu+|D{\bf u}|\right)\operatorname{d}\!x\right)^{p}+C\,\fint_{B_{4R}(x_{0})}\left|{\bf F}-{\bf F}_{0}\right|^{p'}\operatorname{d}\!x
\end{equation}
for any $0<\varepsilon_{3}<1$ and $C=C(\varepsilon_{3},n,p,\nu,L)$.

Finally, the inequalities \eqref{lemu-v-inequ3} and \eqref{lempi-pi_v} reveal that the comparison estimate \eqref{inequ-comparisonlemma1} holds for $1<p<+\infty$\,. Thus the proof of Lemma \ref{approximation1} is completed.
\end{proof}

The second part of this section is to establish a comparison estimate between the symmetric gradient $D{\bf v}$ and the associated pressure $\pi_{\bf v}$ to \eqref{ComparisonSystem} with $D{\bf w}$ and
$\pi_{\bf w}$ to
\eqref{ComparisonSystem-2}\,.

\begin{lemma}\label{approximation2}
Let $\left( {\bf v},\pi_{\bf v}\right)$ be a weak solution pair to \eqref{ComparisonSystem}\,. Then there exist a weak solution pair $\left({\bf w},\pi_{\bf w}\right)$ to \eqref{ComparisonSystem-2} and a positive constant $C=C(n,p,\nu,L)$ such that
\begin{equation}\label{inequ-comparisonlemma2}
\fint_{B_{\frac{3R}{2}}(x_{0})}\left(\left|D{\bf v}-D{\bf w}\right|^{p}+\left|\pi_{\bf v}-\pi_{\bf  w}\right|^{p'}\right)\operatorname{d}\!x\leq C\left[\mathcal{A}\right]_{\operatorname{ BMO}}^{\hat{\sigma}}\left(\frac{3R}{2}\right)\left(\fint_{B_{2R}(x_{0})}\left(\mu+\left|D{\bf v}\right|\right)\operatorname{d}\!x\right)^{p}
\end{equation}
for almost all $x_{0}\in\Omega$ and every $B_{2R}(x_{0})\subset\Omega$, where
\begin{equation}\label{sigma}
\hat{\sigma}=
\left\{\begin{array}{c@{\ \ }ll}
& \displaystyle \frac{p}{2(p-1)}\left(1-\frac{1}{\theta}\right)\,, & \mbox{if}\ \ p\geq2\,, \\[0.1cm]
& \displaystyle \frac{p}{2}\left(1-\frac{1}{\hat{\theta}}\right)\,, & \mbox{if}\ \ 1<p<2\,,
\end{array}\right.
\end{equation}
for some $\theta=\theta(n,\nu,L,p)>p$ and $\hat{\theta}=\min\{\theta,\bar{q}\}$ with $\bar{q}$ defined in \eqref{q}\,.
\end{lemma}

\begin{proof}
A direct computation reveals that $\left( {\bf v}-{\bf w},\pi_{\bf v}-\pi_{\bf w}\right)\in W_{0,\operatorname{div}}^{1,p}(B_{\frac{3R}{2}}(x_{0}), \mathbb{R}^{n})\times L^{p'}(B_{\frac{3R}{2}}(x_{0}))$ is a weak solution pair to
\begin{equation}\label{lemv-w-equ1}
\left\{\begin{array}{r@{\ }c@{\ }ll}
\operatorname{div}\left(\bar{\mathcal{A}}\left(D{\bf v}\right)-\bar{\mathcal{A}}\left(D{\bf w}\right)\right)-\nabla\left(\pi_{\bf v}-\pi_{\bf w}\right)& =&\operatorname{div}\left(\bar{\mathcal{A}}\left(D{\bf v}\right)-\mathcal{A}\left(x,D{\bf v}\right)\right) & \mbox{in}\ \ B_{\frac{3R}{2}}\,, \\[0.05cm]
\operatorname{div}\left({\bf v}- {\bf w}\right)\ &=&\ 0 & \mbox{in}\ \ B_{\frac{3R}{2}}\,, \\[0.05cm]
{\bf v}- {\bf w}\ &=&\ {\bf 0} & \mbox{on}\ \
\partial B_{\frac{3R}{2}}\,,
\end{array}\right.
\end{equation}
where we abbreviate the ball $B_{\frac{3R}{2}}(x_{0})$ to $B_{\frac{3R}{2}}$.
Selecting ${\bf v}- {\bf w}$ as a divergence free test function of \eqref{lemv-w-equ1}, we have
 \begin{eqnarray}\label{v-w-weaksol}
  && \fint_{B_{\frac{3R}{2}}} \left\langle\,\left(\mathcal{A}\left(\cdot,D{\bf v}\right)\right)_{B_{\frac{3R}{2}}}-\left(\mathcal{A}\left(\cdot,D{\bf w}\right)\right)_{B_{\frac{3R}{2}}},D{\bf v}-D{\bf  w}\right\rangle\operatorname{d}\!x\nonumber\\
&=& \fint_{B_{\frac{3R}{2}}}\left\langle\left(\mathcal{A}\left(\cdot,D{\bf v}\right)\right)_{B_{\frac{3R}{2}}}-\mathcal{A}\left(x,D{\bf v}\right)\,,D{\bf v}-D{\bf w}\right\rangle\,\operatorname{d}\!x\,.
 \end{eqnarray}
 Involving the conditions \eqref{structural-cons}, we start by estimating the term on the left side of \eqref{v-w-weaksol} as follows
\begin{eqnarray}\label{v-wleft}
   && \fint_{B_{\frac{3R}{2}}} \left\langle\,\left(\mathcal{A}\left(\cdot,D{\bf v}\right)\right)_{B_{\frac{3R}{2}}}-\left(\mathcal{A}\left(\cdot,D{\bf w}\right)\right)_{B_{\frac{3R}{2}}},D{\bf v}-D{\bf  w}\right\rangle\operatorname{d}\!x \nonumber\\
   &=& \fint_{B_{\frac{3R}{2}}} \fint_{B_{\frac{3R}{2}}}\left\langle\,\mathcal{A}\left(y,D{\bf v}\right)-\mathcal{A}\left(y,D{\bf w}\right),D{\bf v}-D{\bf  w}\right\rangle\operatorname{d}\!y\operatorname{d}\!x \nonumber\\
   &\geq& \nu\fint_{B_{\frac{3R}{2}}}\left(\mu^{2}+|D{\bf v}|^{2}+|D{\bf w}|^{2}\right)^{\frac{p-2}{2}}|D{\bf v}-D{\bf  w}|^{2}\operatorname{d}\!x\,.
\end{eqnarray}
The proof will be divided into two cases that $p\geq2$ and $1<p<2$. For the former situation,
we estimate the term on the right side of \eqref{v-w-weaksol}\,. By Young's inequality, one derives
\begin{eqnarray}\label{v-wright}
   && \fint_{B_{\frac{3R}{2}}}\left\langle\left(\mathcal{A}\left(\cdot,D{\bf v}\right)\right)_{B_{\frac{3R}{2}}}-\mathcal{A}\left(x,D{\bf v}\right)\,,D{\bf v}-D{\bf w}\right\rangle\,\operatorname{d}\!x \nonumber\\
   &\leq& \fint_{B_{\frac{3R}{2}}}\beta\left(\mathcal{A},B_{\frac{3R}{2}}\right)\left(\mu^{2}+|D{\bf v}|^{2}\right)^{\frac{p-1}{2}}\left|D{\bf v}-D{\bf w}\right|\,\operatorname{d}\!x \nonumber\\
   &\leq& C(\hat{\tau}_{1}) \fint_{B_{\frac{3R}{2}}}\beta^{2}\left(\mathcal{A},B_{\frac{3R}{2}}\right)\left(\mu^{2}+|D{\bf v}|^{2}\right)^{\frac{p}{2}}\,\operatorname{d}\!x\nonumber\\
  && +\,\hat{\tau}_{1}\fint_{B_{\frac{3R}{2}}}\left(\mu^{2}+|D{\bf v}|^{2}+|D{\bf w}|^{2}\right)^{\frac{p-2}{2}}\left|D{\bf v}-D{\bf w}\right|^{2}\,\operatorname{d}\!x\,.
\end{eqnarray}
Then combining \eqref{v-wleft} with \eqref{v-wright}, and choosing $\hat{\tau}_{1}=\frac{\nu}{2}$, we derive
\begin{eqnarray}\label{lemv-w-inequ2}
   && \fint_{B_{\frac{3R}{2}}}\left(\mu^{2}+|D{\bf v}|^{2}+|D{\bf w}|^{2}\right)^{\frac{p-2}{2}}|D{\bf v}-D{\bf  w}|^{2}\operatorname{d}\!x\nonumber \\
   &\leq&  C(\nu) \fint_{B_{\frac{3R}{2}}}\beta^{2}\left(\mathcal{A},B_{\frac{3R}{2}}\right)\left(\mu^{2}+|D{\bf v}|^{2}\right)^{\frac{p}{2}}\,\operatorname{d}\!x\,.
\end{eqnarray}
The technique \cite[Theorem 3.4]{DiKaSc} and the known Gehring's Lemma indicate that
the following higher integrability result of $D{\bf v}$
holds for some $\theta=\theta(n,\nu,L,p)>p$ and any $1<p<+\infty$
\begin{equation}\label{v-reverseholder}
  \fint_{B_{\frac{3R}{2}}(x_{0})}|D{\bf v}|^{\theta}\operatorname{d}\!x\leq C\left(\fint_{B_{2R}(x_{0})}\left(\mu^{p}+|D{\bf v}|^{p}\right)\operatorname{d}\!x\right)^{\frac{\theta}{p}}.
\end{equation}
Then by virtue of Lemma \ref{reverse-holder}\,, we have
\begin{equation}\label{reverse-Dv}
  \fint_{B_{\frac{3R}{2}}(x_{0})}|D{\bf v}|^{\theta}\operatorname{d}\!x\leq C\left(\fint_{B_{2R}(x_{0})}\left(\mu+|D{\bf v}|\right)\operatorname{d}\!x\right)^{\theta},
\end{equation}
where $C=C(n,\nu,L,p)$.

Since $p\geq2$, then we apply H\"{o}lder's inequality, the boundedness of $\beta\left(\mathcal{A},B_{\frac{3R}{2}}\right)$ and \eqref{reverse-Dv} to the inequality \eqref{lemv-w-inequ2}, and derive that
\begin{eqnarray}\label{v-w-p>=2}
&&\fint_{B_{\frac{3R}{2}}}\left|D{\bf v}-D{\bf w}\right|^{p}\operatorname{d}\!x \nonumber\\
   &\leq& \fint_{B_{\frac{3R}{2}}}\left(\mu^{2}+|D{\bf v}|^{2}+|D{\bf w}|^{2}\right)^{\frac{p-2}{2}}|D{\bf v}-D{\bf  w}|^{2}\operatorname{d}\!x \nonumber \\
   &\leq&  C(\nu) \fint_{B_{\frac{3R}{2}}}\beta^{2}\left(\mathcal{A},B_{\frac{3R}{2}}\right)\left(\mu^{2}+|D{\bf v}|^{2}\right)^{\frac{p}{2}}\,\operatorname{d}\!x \nonumber\\
    &\leq&  C(\nu) \left(\fint_{B_{\frac{3R}{2}}}\beta^{\frac{2\theta}{\theta-1}}\left(\mathcal{A},B_{\frac{3R}{2}}\right)
    \operatorname{d}\!x\right)^{\frac{\theta-1}{\theta}}
    \left(\fint_{B_{\frac{3R}{2}}}\left(\mu^{2}+|D{\bf v}|^{2}\right)^{\frac{p\theta}{2}}\,\operatorname{d}\!x\right)^{\frac{1}{\theta}} \nonumber\\
    &\leq&  C(n,\nu,L,p) \left(\fint_{B_{\frac{3R}{2}}}\beta\left(\mathcal{A},B_{\frac{3R}{2}}\right)
    \operatorname{d}\!x\right)^{\frac{\theta-1}{\theta}}
    \left(\fint_{B_{\frac{3R}{2}}}\left(\mu^{2}+|D{\bf v}|^{2}\right)^{\frac{p\theta}{2}}\,\operatorname{d}\!x\right)^{\frac{1}{\theta}} \nonumber\\
    &\leq&  C(n,\nu,L,p) \left[\mathcal{A}\right]_{\operatorname {BMO}}^{1-\frac{1}{\theta}}\left(\frac{3R}{2}\right)
    \left(\fint_{B_{2R}}\left(\mu+|D{\bf v}|\right)\operatorname{d}\!x\right)^{p}.
\end{eqnarray}

While for the case $1<p<2$, the inequality \eqref{lemv-w-inequ2} is replaced by
\begin{eqnarray}\label{lemv-w-inequ2-2}
   && \fint_{B_{\frac{3R}{2}}}\left(\mu^{2}+|D{\bf v}|^{2}+|D{\bf w}|^{2}\right)^{\frac{p-2}{2}}|D{\bf v}-D{\bf  w}|^{2}\operatorname{d}\!x\nonumber \\
   &\leq&  C(\nu) \fint_{B_{\frac{3R}{2}}}\beta^{2}\left(\mathcal{A},B_{\frac{3R}{2}}\right)\left(\mu^{2}+|D{\bf v}|^{2}+|D{\bf w}|^{2}\right)^{\frac{p}{2}}\,\operatorname{d}\!x\,.
\end{eqnarray}
Then we shall use the following higher integrability result of $D{\bf w}$ which is introduced in \cite[Lemma 4.3]{BC}
\begin{equation*}\label{w-reverseholder}
  \fint_{B_{R}(x_{0})}|D{\bf w}|^{q}\operatorname{d}\!x\leq C\left(\fint_{B_{\frac{3R}{2}}(x_{0})}\left(\mu^{p}+|D{\bf w}|^{p}\right)\operatorname{d}\!x\right)^{\frac{q}{p}}
\end{equation*}
for all $q\in[p,\bar{q}]$ and $p\in(1,+\infty)$, where
\begin{equation}\label{q}
\bar{q}=
\left\{\begin{array}{c@{\,\ }ll}
& {\rm any \ number\ in\ }(p,+\infty)& \mbox{if}\ \ n=2\,, \\[0.1cm]
& \displaystyle\frac{np}{n-2} & \mbox{if}\ \ n\geq3\,.
\end{array}\right.
\end{equation}
It follows from Lemma \ref{reverse-holder} that
\begin{equation}\label{reverse-Dw}
  \fint_{B_{R}(x_{0})}|D{\bf w}|^{q}\operatorname{d}\!x\leq C\left(\fint_{B_{\frac{3R}{2}}(x_{0})}\left(\mu+|D{\bf w}|\right)\operatorname{d}\!x\right)^{q},
\end{equation}
where $C=C(n,\nu,L,p)$.

Since $1<p<2$, we apply H\"{o}lder's inequality, \eqref{reverse-Dw} as well as Lemma \ref{Dv-control-Dw} to the inequality \eqref{lemv-w-inequ2-2}\,, and then combine with the same argument as \eqref{v-w-p>=2} to deduce that
\begin{eqnarray}\label{v-w-1<p<2}
 && \fint_{B_{\frac{3R}{2}}}\left|D{\bf v}-D{\bf w}\right|^{p}\operatorname{d}\!x \nonumber \\
   &\leq& \left(\fint_{B_{\frac{3R}{2}}}\left(\mu^{2}+|D{\bf v}|^{2}+|D{\bf w}|^{2}\right)^{\frac{p-2}{2}}|D{\bf v}-D{\bf  w}|^{2}\operatorname{d}\!x\right)^{\!\frac{p}{2}}\!\left(\fint_{B_{\frac{3R}{2}}}\left(\mu^{2}+|D{\bf v}|^{2}+|D{\bf w}|^{2}\right)^{\frac{p}{2}}\operatorname{d}\!x\right)^{\frac{2-p}{2}} \nonumber \\
    &\leq&  C(n,\nu,L,p) \left[\mathcal{A}\right]_{\operatorname {BMO}}^{\frac{p}{2}\left(1-\frac{1}{\hat{\theta}}\right)}\left(\frac{3R}{2}\right)
    \left(\fint_{B_{2R}}\left(\mu+|D{\bf v}|\right)\operatorname{d}\!x\right)^{p},
\end{eqnarray}
where $\hat{\theta}=\min\{\theta,\bar{q}\}$.

Hence, a combination of \eqref{v-w-p>=2} and \eqref{v-w-1<p<2} concludes that
\begin{eqnarray}\label{lemv-w-inequ3}
\fint_{B_{\frac{3R}{2}}(x_{0})}\left|D{\bf v}-D{\bf w}\right|^{p}\operatorname{d}\!x
    \leq C\left[\mathcal{A}\right]_{\operatorname {BMO}}^{\sigma}\left(\frac{3R}{2}\right)
    \left(\fint_{B_{2R}(x_{0})}\left(\mu+|D{\bf v}|\right)\operatorname{d}\!x\right)^{p},
\end{eqnarray}
where $C=C(n,\nu,L,p)$ and
\begin{equation*}
\sigma=
\left\{\begin{array}{c@{\,\ }ll}
&\displaystyle 1-\frac{1}{\theta} & \mbox{if}\ \ p\geq2\,, \\[0.1cm]
&\displaystyle \frac{p}{2}\left(1-\frac{1}{\hat{\theta}}\right) & \mbox{if}\ \ 1<p<2\,.
\end{array}\right.
\end{equation*}

In the sequel, it suffices to estimate $\fint_{B_{\frac{3R}{2}}}\left|\pi_{\bf v}-\pi_{\bf w}\right|^{2}\operatorname{d}\!x$. Choosing $\phi\in W_{0}^{1,p}\left(B_{\frac{3R}{2}},\mathbb{R}^{n}\right)$ as a test function of \eqref{lemv-w-equ1}, that is to say
\begin{equation}\label{lemv-w-equ4}
  \fint_{B_{\frac{3R}{2}}}\left(\pi_{\bf v}-\pi_{\bf w}\right)\operatorname{div}\phi \operatorname{d}\!x= \fint_{B_{\frac{3R}{2}}}\big\langle\mathcal{A}\left(x,D{\bf v}\right)-\left(\mathcal{A}\left(\cdot,D{\bf w}\right)\right)_{B_{\frac{3R}{2}}},D\phi\big\rangle \operatorname{d}\!x\,.
\end{equation}
More precisely, selecting the above $\phi$ be a solution to the auxiliary problem
\begin{equation}\label{auxiliaryprob2}
\left\{\begin{array}{r@{\ \ }c@{\ \ }ll}
\operatorname{div}{\bf \phi} & =&\operatorname{sgn}\left(\pi_{\bf v}-\pi_{\bf w} \right)\left|\pi_{v}-\pi_{\bf w} \right|^{\frac{1}{p-1}}
-\left(\operatorname{sgn}\left(\pi_{\bf v}-\pi_{\bf w} \right)\left|\pi_{\bf v}-\pi_{\bf w} \right|^{\frac{1}{p-1}}\right)_{B_{\frac{3R}{2}}} & \mbox{in}\ \ B_{\frac{3R}{2}}\,, \\[0.05cm]
{\bf \phi} &=&  0 & \mbox{on}\ \ \partial B_{\frac{3R}{2}}\,.
\end{array}\right.
\end{equation}
Let
$$h:= \operatorname{sgn}\left(\pi_{\bf v}-\pi_{\bf w} \right)\left|\pi_{v}-\pi_{\bf w} \right|^{\frac{1}{p-1}}
-\left(\operatorname{sgn}\left(\pi_{\bf v}-\pi_{\bf w} \right)\left|\pi_{\bf v}-\pi_{\bf w} \right|^{\frac{1}{p-1}}\right)_{B_{\frac{3R}{2}}}\,.$$
Then it is not difficult to verify that $h\in L^{p}(B_{\frac{3R}{2}})$ and $\int_{B_{\frac{3R}{2}}(x_{0})}h(x)\operatorname{d}\!x=0$.
Hence, the existence of such a solution to auxiliary problem
\eqref{auxiliaryprob2} is ensured by Lemma \ref{existence-Johndomain}
and so
\begin{equation}\label{lemv-w-inequ5}
  \fint_{B_{\frac{3R}{2}}}\left|\nabla\phi\right|^{p}\operatorname{d}\!x\leq C\fint_{B_{\frac{3R}{2}}}\left|\pi_{\bf v}-\pi_{\bf w}\right|^{p'}\operatorname{d}\!x\,,
\end{equation}
where the positive constant $C$ depends only on $n$ and $p$.
Moreover, \cite[Lemma 2.8]{BC} infers that there exists a unique $\pi_{\bf v}-\pi_{\bf w}\in L^{p'}(B_{\frac{3R}{2}}(x_{0}))$ to the problem \eqref{lemv-w-equ1} such that
$\int_{B_{\frac{3R}{2}}}\left(\pi_{\bf v}-\pi_{\bf w}\right)\operatorname{d}\!x=0$.

By substituting such $\phi$ into equality
\eqref{lemv-w-equ4} and combining Young's inequality with
\eqref{lemv-w-inequ3} and \eqref{lemv-w-inequ5}, we obtain
\begin{eqnarray*}
   && \fint_{B_{\frac{3R}{2}}}\left|\pi_{\bf v}-\pi_{\bf w}\right|^{p'} \operatorname{d}\!x\nonumber\\
   &=&\fint_{B_{\frac{3R}{2}}}\left(\pi_{\bf v}-\pi_{\bf w}\right)\left[\operatorname{sgn}\left(\pi_{\bf v}-\pi_{\bf w} \right)\left|\pi_{\bf v}-\pi_{\bf w} \right|^{\frac{1}{p-1}}-\left(\operatorname{sgn}\left(\pi_{\bf v}-\pi_{\bf w} \right)\left|\pi_{\bf v}-\pi_{\bf w} \right|^{\frac{1}{p-1}}\right)_{B_{\frac{3R}{2}}}\right] \operatorname{d}\!x \nonumber\\
   &=& \fint_{B_{\frac{3R}{2}}}\big\langle\mathcal{A}\left(x,D{\bf v}\right)-\left(\mathcal{A}\left(\cdot,D{\bf w}\right)\right)_{B_{\frac{3R}{2}}},D\phi\big\rangle \operatorname{d}\!x \nonumber\\
  &\leq& C(\hat{\tau}_{2},p)\fint_{B_{\frac{3R}{2}}}\left|\mathcal{A}\left(x,D{\bf v}\right)-\left(\mathcal{A}\left(\cdot,D{\bf w}\right)\right)_{B_{\frac{3R}{2}}}\right|^{p'}\operatorname{d}\!x+\hat{\tau}_{2} \fint_{B_{\frac{3R}{2}}}\left|D\phi\right|^{p}\operatorname{d}\!x \\
  &\leq&C(n,p)\,\hat{\tau}_{2} \fint_{B_{\frac{3R}{2}}}\left|\bf\nabla\phi\right|^{p}\operatorname{d}\!x+ C(\hat{\tau}_{2},n,p,\nu,L)\left[\mathcal{A}\right]_{\operatorname {BMO}}^{\sigma}\left(\frac{3R}{2}\right)
    \left(\fint_{B_{2R}}\left(\mu+|D{\bf v}|\right)\operatorname{d}\!x\right)^{p}\\
    &&+\,C(\hat{\tau}_{2},p,L)\fint_{B_{\frac{3R}{2}}}\left[\left(\mu^{2}+\left|D{\bf v}\right|^{2}+\left|D{\bf  w}\right|^{2}\right)^{\frac{p-2}{2}}|D{\bf v}-D{\bf w}|\right]^{p'}\operatorname{d}\!x \\
  &\leq&C_{2}(n,p)\,\hat{\tau}_{2} \fint_{B_{\frac{3R}{2}}}\left|\pi_{\bf v}-\pi_{\bf w}\right|^{p'}\operatorname{d}\!x+ C(\hat{\tau}_{2},n,p,\nu,L)\left[\mathcal{A}\right]_{\operatorname {BMO}}^{\sigma}\left(\frac{3R}{2}\right)
    \left(\fint_{B_{2R}}\left(\mu+|D{\bf v}|\right)\operatorname{d}\!x\right)^{p}\\
    &&+\,C(\hat{\tau}_{2},p,L)\fint_{B_{\frac{3R}{2}}}\left[\left(\mu^{2}+\left|D{\bf v}\right|^{2}+\left|D{\bf  w}\right|^{2}\right)^{\frac{p-2}{2}}|D{\bf v}-D{\bf w}|\right]^{p'}\operatorname{d}\!x\,.
\end{eqnarray*}
Choosing the positive constant $\hat{\tau}_{2}$ sufficiently small
such that $C_{2}\hat{\tau}_{2}=\frac{1}{2}$\,, we derive that
\begin{eqnarray}\label{lemv-w-inequ6}
 &&\fint_{B_{\frac{3R}{2}}}\left|\pi_{\bf v}-\pi_{\bf w}\right|^{p'} \operatorname{d}\!x \\
 &\leq& C\left[\mathcal{A}\right]_{\operatorname {BMO}}^{\sigma}\left(\frac{3R}{2}\right)
    \left(\fint_{B_{2R}}\!\left(\mu+|D{\bf v}|\right)\operatorname{d}\!x\right)^{p}
    \!+C\fint_{B_{\frac{3R}{2}}}\!\left[\left(\mu^{2}+\left|D{\bf v}\right|^{2}+\left|D{\bf  w}\right|^{2}\right)^{\frac{p-2}{2}}|D{\bf v}-D{\bf w}|\right]^{p'}\!\operatorname{d}\!x\,,\nonumber
\end{eqnarray}
where $C=C(n,p,\nu,L)$.

In order to estimate the last term on the right side of \eqref{lemv-w-inequ6}, we proceed in two situations. Regarding to the case of $1<p<2$, we have
\begin{eqnarray}\label{lemv-w-inequ7}
 &&\fint_{B_{\frac{3R}{2}}}\left|\pi_{\bf v}-\pi_{\bf w}\right|^{p'} \operatorname{d}\!x\nonumber\\
 &\leq& C\left[\mathcal{A}\right]_{\operatorname {BMO}}^{\frac{p}{2}\left(1-\frac{1}{\hat{\theta}}\right)}\left(\frac{3R}{2}\right)
    \left(\fint_{B_{2R}}\left(\mu+|D{\bf v}|\right)\operatorname{d}\!x\right)^{p}
    +C\fint_{B_{\frac{3R}{2}}}|D{\bf v}-D{\bf w}|^{p}\operatorname{d}\!x \nonumber\\
    &\leq& C(n,p,\nu,L)\left[\mathcal{A}\right]_{\operatorname {BMO}}^{\frac{p}{2}\left(1-\frac{1}{\hat{\theta}}\right)}\left(\frac{3R}{2}\right)
    \left(\fint_{B_{2R}}\left(\mu+|D{\bf v}|\right)\operatorname{d}\!x\right)^{p},
\end{eqnarray}
where the last inequality is ensured by \eqref{lemv-w-inequ3}.

While for the case of $p\geq2$, we use H\"{o}lder's inequality, \eqref{v-w-p>=2}, Lemma \ref{Dv-control-Dw} and the reserve H\"{o}lder's inequality \eqref{reverse-Dv} to deduce that
\begin{eqnarray}\label{lemv-w-inequ8}
 &&\fint_{B_{\frac{3R}{2}}}\left[\left(\mu^{2}+\left|D{\bf v}\right|^{2}+\left|D{\bf  u}\right|^{2}\right)^{\frac{p-2}{2}}|D{\bf v}-D{\bf w}|\right]^{p'}\operatorname{d}\!x \nonumber\\
 &\leq&\left(\fint_{B_{\frac{3R}{2}}}\left(\mu^{2}+\left|D{\bf v}\right|^{2}+\left|D{\bf  w}\right|^{2}\right)^{\frac{p-2}{2}}|D{\bf v}-D{\bf w}|^{2}\operatorname{d}\!x\right)^{\frac{p}{2(p-1)}}\nonumber\\
 &&\times\left(\fint_{B_{\frac{3R}{2}}}\left(\mu^{2}+\left|D{\bf v}\right|^{2}+\left|D{\bf  w}\right|^{2}\right)^{\frac{p}{2}}\operatorname{d}\!x\right)^{\frac{p-2}{2(p-1)}}\nonumber\\
 &\leq&C(n,p,\nu,L)\left[\mathcal{A}\right]_{\operatorname {BMO}}^{\frac{p}{2(p-1)}\left(1-\frac{1}{\theta}\right)}\left(\frac{3R}{2}\right)
    \left(\fint_{B_{2R}}\left(\mu+|D{\bf v}|\right)\operatorname{d}\!x\right)^{p}.
\end{eqnarray}
Inserting \eqref{lemv-w-inequ8} into \eqref{lemv-w-inequ6}, it follows from $p\geq2$ that
\begin{eqnarray}\label{lemv-w-inequ9}
 &&\fint_{B_{\frac{3R}{2}}}\left|\pi_{\bf v}-\pi_{\bf w}\right|^{p'} \operatorname{d}\!x\nonumber\\
 &\leq& C\left[\mathcal{A}\right]_{\operatorname {BMO}}^{1-\frac{1}{\theta}}\left(\frac{3R}{2}\right)
    \left(\fint_{B_{2R}}\left(\mu+|D{\bf v}|\right)\operatorname{d}\!x\right)^{p}
    +C\left[\mathcal{A}\right]_{\operatorname {BMO}}^{\frac{p}{2(p-1)}\left(1-\frac{1}{\theta}\right)}\left(\frac{3R}{2}\right)
    \left(\fint_{B_{2R}}\left(\mu+|D{\bf v}|\right)\operatorname{d}\!x\right)^{p} \nonumber\\
 &\leq& C\left[\mathcal{A}\right]_{\operatorname {BMO}}^{\frac{p}{2(p-1)}\left(1-\frac{1}{\theta}\right)}\left(\frac{3R}{2}\right)
    \left(\fint_{B_{2R}}\left(\mu+|D{\bf v}|\right)\operatorname{d}\!x\right)^{p}.
\end{eqnarray}

Finally, a combination of \eqref{lemv-w-inequ3}, \eqref{lemv-w-inequ7} with \eqref{lemv-w-inequ9} yields that \eqref{inequ-comparisonlemma2} is true for $1<p<+\infty$. Then the proof of Lemma \ref{approximation2} is completed.
\end{proof}

In the end of this section, our intention is to establish the comparison estimate between $D{\bf u}$, $\pi$ with $D{\bf w}$, $\pi_{\bf w}$ by using Lemma \ref{approximation1}\,, Lemma \ref{approximation2}\,, H\"{o}lder's inequality, and combining Lemma \ref{Du-control-Dv} with Lemma \ref{reserve_Holder_Du}\,.
\begin{lemma}\label{Lemma-comparison}
Let $\left( {\bf u},\pi\right)$ be a weak solution pair to \eqref{model}\,. Then there exists a weak solution pair $\left({\bf w},\pi_{\bf w}\right)$ to \eqref{ComparisonSystem-2} such that
\begin{eqnarray}\label{inequ-comparison}
&& \fint_{B_{\frac{3R}{2}}(x_{0})}\left(\left|D{\bf u}-D{\bf w}\right|^{p}+\left|\pi-\pi_{\bf  w}\right|^{p'}\right)\operatorname{d}\!x \\
&\leq& C_{1}\left(\varepsilon\chi_{\{p\neq2\}}+\left[\mathcal{A}\right]_{\operatorname{ BMO}}^{\hat{\sigma}}\left(\frac{3R}{2}\right)\right)\left(\fint_{B_{4R}(x_{0})}\left(\mu+\left|D{\bf u}\right|\right)\operatorname{d}\!x\right)^{p}\nonumber\\
&&+\,C_{2}\left(1+\left[\mathcal{A}\right]_{\operatorname{ BMO}}^{\hat{\sigma}}\left(\frac{3R}{2}\right)\right)\fint_{B_{4R}(x_{0})}\!\left|{\bf F}-{\bf F}_{0}\right|^{p'}\operatorname{d}\!x\nonumber
\end{eqnarray}
for any $0<\varepsilon<1$, any constant matrix ${\bf F}_{0}\in \mathbb{R}_{\rm sym}^{n\times n}$, almost all $x_{0}\in\Omega$ and every $B_{4R}(x_{0})\subset\Omega$\,. Here the positive constants $C_{1}=C_{1}(n,p,\nu,L)$, $C_{2}=C_{2}(\varepsilon,n,p,\nu,L)$\,, and $\hat{\sigma}$ is given in \eqref{sigma}.
\end{lemma}

\section{Nonlinear potential estimates.}\label{section4}

In this section, we first establish the pointwise gradient estimate in Theorem \ref{Th1}\,. Before proceeding further, we need  to prove the following Campanato type decay estimate for the shear rate $D\bf w$ to \eqref{ComparisonSystem-2} in the plane.
\begin{lemma}\label{lem-w-decayestimate}
Let ${\bf w}$ be the weak solution to \eqref{ComparisonSystem-2} and the dimension $n=2$\,. Then there exist positive constants $\alpha\in(0,1]$ depending only on $p$, $\nu$ and $L$, such that the estimate
\begin{eqnarray}\label{w-decayestimate}
 &&\fint_{B_{\rho R}(x_{0})}\left|D{\bf w}-\left(D{\bf w}\right)_{B_{\rho R}(x_{0})}\right|\operatorname{d}\!x \nonumber\\
 &\leq& \varepsilon\chi_{\{p\neq2\}}\fint_{B_{\frac{3R}{2}}(x_{0})}\left(\mu+\left|D{\bf w}\right|\right)\operatorname{d}\!x+C\,\rho^{\alpha} \fint_{B_{\frac{3R}{2}}(x_{0})}\left|D{\bf w}-\left(D{\bf w}\right)_{B_{\frac{3R}{2}}(x_{0})}\right|\operatorname{d}\!x
\end{eqnarray}
holds for any $\rho\in (0,\frac{3}{2}]$ and $\varepsilon\in(0,1)$, where $C=C(\varepsilon,p,\nu,L)$.
\end{lemma}
\begin{proof}
We may assume without loss of generality that $\rho\in (0,1]$, since \eqref{lem-w-decayestimate} obviously holds for $1<\rho\leq\frac{3}{2}$. By virtue of Theorem 3.8 in \cite{DiKaSc}, we have
\begin{eqnarray*}
\fint_{B_{\rho R}(x_{0})}\left|V_{p}(D{\bf w})-\left(V_{p}(D{\bf w})\right)_{B_{\rho R}(x_{0})}\right|^{2}\operatorname{d}\!x
  \leq C(p,\nu,L)\rho^{p\alpha}\fint_{B_{R}(x_{0})}\left|V_{p}(D{\bf w})-\left(V_{p}(D{\bf w})\right)_{B_{R}(x_{0})}\right|^{2}\operatorname{d}\!x\,,
\end{eqnarray*}
where $V_{p}(\xi)=\left(\mu^{2}+|\xi|^{2}\right)^{\frac{p-2}{4}}\xi$.
And then, using \cite[Lemma 6.2]{DKS}, the above inequality is equivalent to
\begin{eqnarray*}
&&\fint_{B_{\rho R}(x_{0})}\left|V_{p}(D{\bf w})-V_{p}\left(\left(D{\bf w}\right)_{B_{\rho R}(x_{0})}\right)\right|^{2}\operatorname{d}\!x\\
  &\leq& C(p,\nu,L)\rho^{p\alpha}\fint_{B_{R}(x_{0})}\left|V_{p}(D{\bf w})-V_{p}\left(\left(D{\bf w}\right)_{B_{R}(x_{0})}\right)\right|^{2}\operatorname{d}\!x\,,
\end{eqnarray*}
Now dividing into two cases that $p\geq2$ and $1<p<2$\,, and proceeding similarly as before, we obtain
\begin{eqnarray}\label{Dw-decayestimate}
   && \fint_{B_{\rho R}(x_{0})}\left|D{\bf w}-\left(D{\bf w}\right)_{B_{\rho R}(x_{0})}\right|^{p}\operatorname{d}\!x\nonumber \\
   &\leq&  \varepsilon_{1}\chi_{\{p\neq2\}}\fint_{B_{R}(x_{0})}\left(\mu^{p}+\left|D{\bf w}\right|^{p}\right)\operatorname{d}\!x+C(\varepsilon_{1},p,\nu,L)\rho^{p\alpha}\fint_{B_{R}(x_{0})}\left|D{\bf w}-\left(D{\bf w}\right)_{B_{R}(x_{0})}\right|^{p}\operatorname{d}\!x\,.
\end{eqnarray}
Next, the approach of estimate for the last term on the right side of \eqref{Dw-decayestimate} is to use the reserve H\"{o}lder inequality introduced in \cite{DK} as follows
\begin{eqnarray*}
 && \fint_{B_{ R}(x_{0})}\left|V_{p}(D{\bf w})-V_{p}\left(\left(D{\bf w}\right)_{B_{R}(x_{0})}\right)\right|^{2}\operatorname{d}\!x\\
  &\leq&C(p,\nu,L)\left(\fint_{B_{\frac{3R}{2}}(x_{0})}\left|V_{p}(D{\bf w})-V_{p}\left(\left(D{\bf w}\right)_{B_{\frac{3R}{2}}(x_{0})}\right)\right|^{\frac{2}{p}}\operatorname{d}\!x\right)^{p}\,.
\end{eqnarray*}
The following discussion is still divided into $p\geq2$ and $1<p<2$, then we can derive
\begin{eqnarray}\label{reverse_Dw-(Dw)}
   && \fint_{B_{ R}(x_{0})}\left|D{\bf w}-\left(D{\bf w}\right)_{B_{ R}(x_{0})}\right|^{p}\operatorname{d}\!x\nonumber \\
   &\leq&  \varepsilon_{2}\chi_{\{p\neq2\}}\left(\fint_{B_{\frac{3R}{2}}(x_{0})}\left(\mu+\left|D{\bf w}\right|\right)\operatorname{d}\!x\right)^{p}+C(\varepsilon_{2},\nu,L,p)\left(\fint_{B_{\frac{3R}{2}}(x_{0})}\left|D{\bf w}-\left(D{\bf w}\right)_{B_{\frac{3R}{2}}(x_{0})}\right|\operatorname{d}\!x\right)^{p}\,.\nonumber\\
\end{eqnarray}
Finally, inserting \eqref{reverse_Dw-(Dw)} into \eqref{Dw-decayestimate} and combining H\"{o}lder's inequality with the reverse H\"{o}lder inequality \eqref{reverse-Dw} for $D{\bf w}$, we conclude that
\begin{eqnarray*}
 &&\fint_{B_{\rho R}(x_{0})}\left|D{\bf w}-\left(D{\bf w}\right)_{B_{\rho R}(x_{0})}\right|\operatorname{d}\!x \nonumber\\
 &\leq& \varepsilon\chi_{\{p\neq2\}}\fint_{B_{\frac{3R}{2}}(x_{0})}\left(\mu+\left|D{\bf w}\right|\right)\operatorname{d}\!x+C\,\rho^{\alpha} \fint_{B_{\frac{3R}{2}}(x_{0})}\left|D{\bf w}-\left(D{\bf w}\right)_{B_{\frac{3R}{2}}(x_{0})}\right|\operatorname{d}\!x
\end{eqnarray*}
is valid for any $\rho\in (0,1]$ and $\varepsilon\in(0,1)$, where $C=C(\varepsilon,p,\nu,L)$.
Thus, the proof of Lemma \ref{lem-w-decayestimate} is completed.
\end{proof}

\begin{remark}
Due to the absence of Lipschitz regularity for solution to the corresponding limiting problem \eqref{ComparisonSystem-2} in higher dimensional space, the above Campanato type decay estimate holds only in the plane. This is the immediate trigger for the pointwise gradient estimate established only to the planar flows.
\end{remark}

We now turn our attention to the Campanato type decay estimate of $D{\bf u}$, which is the main ingredient to carry on the proof of Theorem \ref{Th1}\,.
\begin{lemma}\label{lem-u-decayestimate}
Let $\beta\in(0,1)$ and $\left({\bf u},\pi\right)$ be a weak
solution pair to \eqref{model} with ${\bf F}\in
L^{p'}_{\rm loc}(\Omega,\mathbb{R}^{2\times 2}_{\rm sym})$. There exists a positive
constant $\rho=\rho(p,\nu,L,\beta)\in\left(0,\frac{1}{4}\right]$ such that
\begin{eqnarray}\label{u-decayestimate}
&& \left(\fint_{B_{\rho R}(x_{0})}\left|D{\bf
u}-\left(D{\bf
u}\right)_{B_{\rho R}(x_{0})}\right|\operatorname{d}\!x\right)^{p}+
\left(\fint_{B_{\rho R}(x_{0})}\left|{\pi}-\left(\pi\right)_{B_{\rho R}(x_{0})}\right|\operatorname{d}\!x\right)^{p'}\nonumber\\
&\leq &\beta C_{1}\left(\fint_{B_{R}(x_{0})}\left|D{\bf
u}-\left(D {\bf
u}\right)_{B_{R}(x_{0})}\right|\operatorname{d}\!x\right)^{p}
 + C_{2}\fint_{B_{R}(x_{0})}\left|{\bf F}-{\bf F}_{0}\right|^{p'}\operatorname{d}\!x\nonumber\\
&&+\,C_{3}\left(\left[\mathcal{A}\right]_{\operatorname{ BMO}}^{\hat{\sigma}}(R)+\varepsilon\chi_{\{p\neq2\}}\right)\left(\fint_{B_{R}(x_{0})}\left(\mu+\left|D{\bf
u}\right|\right)\operatorname{d}\!x\right)^{p}
\end{eqnarray}
for any $\varepsilon\in(0,1)$, any constant matrix ${\bf F}_{0}\in \mathbb{R}_{\rm sym}^{2\times 2}$ and every $B_{R}(x_{0})\subset\Omega$, where the positive constants $C_{1}=C_{1}(\varepsilon,\nu,L,p)$\,, $C_{2}=C_{2}(\varepsilon,\nu,L,p,\beta)$\,, $C_{3}=C_{3}(\nu,L,p,\beta)$\,, and $\hat{\sigma}$ is given in \eqref{sigma}\,.
\end{lemma}

\begin{proof}
In order to prove this technical lemma, we apply H\"{o}lder's inequality with \eqref{eqn-minimal2} to obtain
\begin{eqnarray}\label{lem-u-decay-inequ1}
&& \left(\fint_{B_{\rho R}(x_{0})}\left|D{\bf
u}-\left(D{\bf
u}\right)_{B_{\rho R}(x_{0})}\right|\operatorname{d}\!x\right)^{p}+
\left(\fint_{B_{\rho R}(x_{0})}\left|{\pi}-\left(\pi\right)_{B_{\rho R}(x_{0})}\right|
\operatorname{d}\!x\right)^{p'}\nonumber\\
&\leq&\frac{C(p)}{\rho^{2p}}\left[\fint_{B_{\frac{3R}{2}}(x_{0})}\left(\left|D{\bf
u}-D{\bf
w}\right|^{p}+\left|{\pi}-\pi_{\bf w}\right|^{p'}\right)\operatorname{d}\!x\right]+C(p)\left(\fint_{B_{\rho R}(x_{0})}\left|D{\bf
w}-\left(D{\bf
w}\right)_{B_{\rho R}(x_{0})}\right|\operatorname{d}\!x\right)^{p}\nonumber\\
&&+\,C(p)\left(\fint_{B_{\rho R}(x_{0})}
\left|\pi_{\bf w}-\left(\pi_{\bf w}\right)_{B_{\rho R}(x_{0})}\right|\operatorname{d}\!x\right)^{p'}
\end{eqnarray}
for any $0<\rho\leq1$.

We first estimate the last term on the right side
of \eqref{lem-u-decay-inequ1}\,. Let $\psi\in W_{0}^{1,p}\left(B_{\rho R}(x_{0}),\mathbb{R}^{2}\right)$ be a test function of \eqref{ComparisonSystem-2}, i.e.,
\begin{eqnarray}\label{lem-u-decay-equ2}
 && \fint_{B_{\rho R}(x_{0})}\left(\pi_{\bf w}-\left(\pi_{\bf w}\right)_{B_{\rho R}(x_{0})}\right)\operatorname{div}\psi \operatorname{d}\!x\nonumber\\
 &=& \fint_{B_{\rho R}(x_{0})}\big\langle\left(\mathcal{A}\left(\cdot,D{\bf w}\right)\right)_{B_{\frac{3R}{2}}(x_{0})}-\left(\mathcal{A}\left(\cdot,\left(D{\bf w}\right)_{B_{\rho R}(x_{0})}\right)\right)_{B_{\frac{3R}{2}}(x_{0})},D\psi\big\rangle \operatorname{d}\!x\,.
\end{eqnarray}
More precisely, we select the above $\psi$ be a solution to the following auxiliary problem
\begin{equation}\label{auxiliaryprob3}
\left\{\begin{array}{r@{\ \ }c@{\ \ }ll}
\operatorname{div}{\bf \psi} & =&\operatorname{sgn}\left(\pi_{\bf w}-\left(\pi_{\bf w}\right)_{B_{\rho R}(x_{0})} \right)\left|\pi_{w}-\left(\pi_{\bf w}\right)_{B_{\rho R}(x_{0})}\right|^{\frac{1}{p-1}}\\
&&
-\left(\operatorname{sgn}\left(\pi_{\bf w}-\left(\pi_{\bf w}\right)_{B_{\rho R}(x_{0})} \right)\left|\pi_{\bf w}-\left(\pi_{\bf w}\right)_{B_{\rho R}(x_{0})} \right|^{\frac{1}{p-1}}\right)_{B_{\rho R}(x_{0})}
 & \mbox{in}\ \ B_{\rho R}(x_{0})\,, \\[0.05cm]
{\bf \psi} &=&  0 & \mbox{on}\ \ \partial B_{\rho R}(x_{0})\,,
\end{array}\right.
\end{equation}
where the nonhomogeneous term belongs to $L^{p}(B_{\rho R}(x_{0}))$ and satisfies
\begin{eqnarray*}
  &\displaystyle\int_{B_{\rho R}(x_{0})}&\operatorname{sgn}\left(\pi_{\bf w}-\left(\pi_{\bf w}\right)_{B_{\rho R}(x_{0})} \right)\left|\pi_{w}-\left(\pi_{\bf w}\right)_{B_{\rho R}(x_{0})}\right|^{\frac{1}{p-1}} \\
  &&  -\left(\operatorname{sgn}\left(\pi_{\bf w}-\left(\pi_{\bf w}\right)_{B_{\rho R}(x_{0})} \right)\left|\pi_{\bf w}-\left(\pi_{\bf w}\right)_{B_{\rho R}(x_{0})} \right|^{\frac{1}{p-1}}\right)_{B_{\rho R}(x_{0})}\operatorname{d}\!x=0\,.
\end{eqnarray*}
Then Lemma \ref{existence-Johndomain} infers that there exists a
 solution to auxiliary problem
\eqref{auxiliaryprob3} such that
\begin{equation}\label{lem-u-decay-inequ3}
  \left\|\bf\nabla\psi\right\|_{L^{p}\left(B_{\rho R}(x_{0})\right)}\leq C(p)\left\|\pi_{\bf w}-\left(\pi_{\bf w}\right)_{B_{\rho R}(x_{0})}\right\|_{L^{p'}\left(B_{\rho R}(x_{0})\right)}.
\end{equation}
Substituting such $\psi$ into equality
\eqref{lem-u-decay-equ2} and combining Young's inequality with \eqref{structural-cons}, \eqref{lem-u-decay-inequ3} and reserve H\"{o}lder type inequality \eqref{reverse-Dw}\,, we deduce that
\begin{eqnarray*}
   && \fint_{B_{\rho R}(x_{0})}\left|\pi_{\bf w}-\left(\pi_{\bf w}\right)_{B_{\sigma\!R}(x_{0})}\right|^{p'} \operatorname{d}\!x\nonumber\\
  &\leq& C(\tau,L,p)\fint_{B_{\rho R}(x_{0})}\left[\left(\mu^{2}+\left|D{\bf
w}\right|^{2}+\left|\left(D{\bf w}\right)_{B_{\rho R}(x_{0})}\right|^{2}\right)^{\frac{p-2}{2}}\left|D{\bf w}-\left(D{\bf w}\right)_{B_{\rho R}(x_{0})}\right|\right]^{p'}\operatorname{d}\!x\\
&&+\,\tau \fint_{B_{\rho R}(x_{0})}\left|\nabla\psi\right|^{p}\operatorname{d}\!x \\
&\leq& \varepsilon\chi_{\{p>2\}}\fint_{B_{\rho R}(x_{0})}\left(\mu^{p}+\left|D{\bf
w}\right|^{p}\right)\operatorname{d}\!x+C(\varepsilon,\tau,L,p)\fint_{B_{\rho R}(x_{0})}\left|D{\bf w}-\left(D{\bf w}\right)_{B_{\rho R}(x_{0})}\right|^{p}\operatorname{d}\!x\\
&&+\,C(p)\,\tau \fint_{B_{\rho R}(x_{0})}\left|\pi_{\bf w}-\left(\pi_{\bf w}\right)_{B_{\rho R}(x_{0})}\right|^{p'} \operatorname{d}\!x \\
   &\leq& C(\nu,L,p)\varepsilon\chi_{\{p>2\}}\left(\fint_{B_{\frac{3}{2}\rho R}(x_{0})}\left(\mu+\left|D{\bf
w}\right|\right)\operatorname{d}\!x\right)^{p}+C(\varepsilon,\tau,L,p)\fint_{B_{\rho R}(x_{0})}\left|D{\bf w}-\left(D{\bf w}\right)_{B_{\rho R}(x_{0})}\right|^{p}\operatorname{d}\!x\\
&&+\,C(p)\,\tau \fint_{B_{\rho R}(x_{0})}\left|\pi_{\bf w}-\left(\pi_{\bf w}\right)_{B_{\rho R}(x_{0})}\right|^{p'} \operatorname{d}\!x\,.
\end{eqnarray*}
Selecting the positive constant $\tau=\frac{1}{2C(p)}$, we apply H\"{o}lder's inequality and \eqref{reverse_Dw-(Dw)} to derive that
\begin{eqnarray}\label{lem-u-decay-inequ4}
   &&\fint_{B_{\rho R}(x_{0})}\left|\pi_{\bf w}-\left(\pi_{\bf w}\right)_{B_{\rho R}(x_{0})}\right|^{p'} \operatorname{d}\!x\nonumber\\
   &\leq& C(\nu,L,p)\varepsilon\chi_{\{p>2\}}\left(\fint_{B_{\frac{3}{2}\rho R}(x_{0})}\left(\mu+\left|D{\bf
w}\right|\right)\operatorname{d}\!x\right)^{p}+C(\varepsilon,L,p)\fint_{B_{\rho R}(x_{0})}\left|D{\bf w}-\left(D{\bf w}\right)_{B_{\rho R}(x_{0})}\right|^{p}\operatorname{d}\!x \nonumber\\
&\leq& C(\nu,L,p)\varepsilon\chi_{\{p>2\}}\left[\rho^{-2p}\fint_{B_{\frac{3R}{2}}(x_{0})}\left(\mu^{p}+\left|D{\bf u}-D{\bf
w}\right|^{p}\right)\operatorname{d}\!x+\left(\fint_{B_{\frac{3}{2}\rho R}(x_{0})}\left|D{\bf
u}\right|\operatorname{d}\!x\right)^{p}\,\right]\nonumber\\
&&+\,C(\varepsilon,L,p)\fint_{B_{\rho R}(x_{0})}\left|D{\bf w}-\left(D{\bf w}\right)_{B_{\rho R}(x_{0})}\right|^{p}\operatorname{d}\!x\nonumber\\
&\leq& C(\nu,L,p)\varepsilon\chi_{\{p\neq2\}}\left[\,\rho^{-2p}\fint_{B_{\frac{3R}{2}}(x_{0})}\left(\mu^{p}+\left|D{\bf u}-D{\bf
w}\right|^{p}\right)\operatorname{d}\!x+\left(\fint_{B_{\frac{3}{2}\rho R}(x_{0})}\left|D{\bf
u}\right|\operatorname{d}\!x\right)^{p}\,\right]\nonumber\\
&&+\,C(\varepsilon,\nu,L,p)\left(\fint_{B_{\frac{3}{2}\rho R}(x_{0})}\left|D{\bf w}-\left(D{\bf w}\right)_{B_{\frac{3}{2}\rho R}(x_{0})}\right|\operatorname{d}\!x\right)^{p}\,.
\end{eqnarray}

Next, inserting \eqref{lem-u-decay-inequ4} into \eqref{lem-u-decay-inequ1}\,, and combining Lemma \ref{lem-w-decayestimate} with Lemma \ref{Lemma-comparison} together with the non-decreasing function $\left[\mathcal{A}\right]_{\operatorname{ BMO}}(\cdot)$, we obtain that
\begin{eqnarray}\label{lem-u-decay-inequ5}
 &&  \left(\fint_{B_{\rho R}(x_{0})}\left|D{\bf
u}-\left(D{\bf
u}\right)_{B_{\rho R}(x_{0})}\right|\operatorname{d}\!x\right)^{p}+
\left(\fint_{B_{\rho R}(x_{0})}\left|{\pi}-\left(\pi\right)_{B_{\rho R}(x_{0})}\right|
\operatorname{d}\!x\right)^{p'}\nonumber\\
&\leq&C\rho^{-2p}\fint_{B_{\frac{3R}{2}}(x_{0})}\left(\left|D{\bf
u}-D{\bf
w}\right|^{p}+\left|{\pi}-\pi_{\bf w}\right|^{p'}\right)\operatorname{d}\!x+C\varepsilon\chi_{\{p\neq2\}}\rho^{-2p}\left(\fint_{B_{2R}(x_{0})}\left(\mu+\left|D{\bf
u}\right|\right)\operatorname{d}\!x\right)^{p}\nonumber\\
&&+\,C\rho^{\alpha p}\left(\fint_{B_{\frac{3R}{2}}(x_{0})}\left|D{\bf
w}-\left(D{\bf
w}\right)_{B_{\frac{3R}{2}}(x_{0})}\right|\operatorname{d}\!x\right)^{p} \nonumber\\
&\leq&C\rho^{\alpha p}\left(\fint_{B_{\frac{3R}{2}}(x_{0})}\left|D{\bf
u}-\left(D{\bf
u}\right)_{B_{\frac{3R}{2}}(x_{0})}\right|\operatorname{d}\!x\right)^{p}+\hat{C}_{\rho}\left(1+\left[\mathcal{A}\right]_{\operatorname{ BMO}}^{\hat{\sigma}}\left(\frac{3R}{2}\right)\right)\fint_{B_{4R}(x_{0})}\!\left|{\bf F}-{\bf F}_{0}\right|^{p'}\operatorname{d}\!x\nonumber\\
&&+\,C_{\rho}\left(\left[\mathcal{A}\right]_{\operatorname{ BMO}}^{\hat{\sigma}}\left(\frac{3R}{2}\right)+\varepsilon\chi_{\{p\neq2\}}\right)\left(\fint_{B_{4R}(x_{0})}
\left(\mu+\left|D{\bf u}\right|\right)\operatorname{d}\!x\right)^{p} \nonumber\\
&\leq&C\rho^{\alpha p}\left(\fint_{B_{4R}(x_{0})}\left|D{\bf
u}-\left(D{\bf
u}\right)_{B_{4R}(x_{0})}\right|\operatorname{d}\!x\right)^{p}+\hat{C}_{\rho}\left(1+\left[\mathcal{A}\right]_{\operatorname{ BMO}}^{\hat{\sigma}}(4R)\right)\fint_{B_{4R}(x_{0})}\!\left|{\bf F}-{\bf F}_{0}\right|^{p'}\operatorname{d}\!x\nonumber\\
&&+\,C_{\rho}\left(\left[\mathcal{A}\right]_{\operatorname{ BMO}}^{\hat{\sigma}}(4R)+\varepsilon\chi_{\{p\neq2\}}\right)\left(\fint_{B_{4R}(x_{0})}\left(\mu+\left|D{\bf
u}\right|\right)\operatorname{d}\!x\right)^{p}
\end{eqnarray}
for any $0<\rho\leq1$ and $0<\varepsilon<1$, where the positive constants $C=C(\varepsilon,p,\nu,L)$, $C_{\rho}=C_{\rho}(\rho,p,\nu,L)$ and $\hat{C}_{\rho}=\hat{C}_{\rho}(\varepsilon,\rho,p,\nu,L)$.
It is plain that \eqref{lem-u-decay-inequ5} is equivalent to the inequality
\begin{eqnarray*}
&&\left(\fint_{B_{\rho R}(x_{0})}\left|D{\bf
u}-\left(D{\bf
u}\right)_{B_{\rho R}(x_{0})}\right|\operatorname{d}\!x\right)^{p}+
\left(\fint_{B_{\rho R}(x_{0})}\left|{\pi}-\left(\pi\right)_{B_{\rho R}(x_{0})}\right|
\operatorname{d}\!x\right)^{p'}\\
 &\leq&C\rho^{\alpha p}\left(\fint_{B_{R}(x_{0})}\left|D{\bf
u}-\left(D{\bf
u}\right)_{B_{R}(x_{0})}\right|\operatorname{d}\!x\right)^{p}+\hat{C}_{\rho}\left(1+\left[\mathcal{A}\right]_{\operatorname{ BMO}}^{\hat{\sigma}}(R)\right)\fint_{B_{R}(x_{0})}\!\left|{\bf F}-{\bf F}_{0}\right|^{p'}\operatorname{d}\!x\nonumber\\
&&+\,C_{\rho}\left(\left[\mathcal{A}\right]_{\operatorname{ BMO}}^{\hat{\sigma}}(R)+\varepsilon\chi_{\{p\neq2\}}\right)\left(\fint_{B_{R}(x_{0})}\left(\mu+\left|D{\bf
u}\right|\right)\operatorname{d}\!x\right)^{p}
\end{eqnarray*}
for any $0<\rho\leq\frac{1}{4}$ and $0<\varepsilon<1$.

Finally, selecting $\rho$ small enough such that $\rho^{\alpha p}\leq\beta$ for any $\beta\in(0,1)$, the above inequality implies that the desired decay estimate \eqref{u-decayestimate} holds for any $B_{R}(x_{0})\subset\Omega$. Thus, the proof of Lemma \ref{lem-u-decayestimate} is completed.
\end{proof}

Based on the above preparations, we are now in a position to prove Theorem \ref{Th1}\,.

\begin{proof}[Proof of Theorem \ref{Th1}] Without loss of generality, we may assume that
\begin{equation*}
  \int_{0}^{2R}\left( \fint_{B_{\varrho}(x_{0})}\left|{\bf F}-({\bf F})_{B_{\varrho}(x_{0})}\right|^{p'}\operatorname{d}\!x\right)^{\frac{1}{p}}\frac{\operatorname{d}\!\varrho}{\varrho}<\infty\,,
\end{equation*}
otherwise \eqref{gradient-estimate} is trivial. Through a direct calculation, one has
\begin{eqnarray}\label{th1-inequ1}
  \left|\fint_{B_{\rho^{k}R}(x_{0})}D{\bf u}\operatorname{d}\!x-\fint_{B_{R}(x_{0})}D{\bf u}\operatorname{d}\!x \right| &=& \left|\sum_{i=0}^{k-1}\left(\fint_{B_{\rho^{i+1}R}(x_{0})}D{\bf u}\operatorname{d}\!x-\fint_{B_{\rho^{i}R}(x_{0})}D{\bf u}\operatorname{d}\!x\right) \right| \nonumber \\
  &\leq& \rho^{-2}\sum_{i=0}^{k-1}\left|\fint_{B_{\rho^{i}R}(x_{0})}\left(D{\bf u}-\left(D{\bf u}\right)_{B_{\rho^{i}R}(x_{0})}\right)\operatorname{d}\!x \right|
\end{eqnarray}
for any $k\in\mathbb{N}$. Similarly, the following inequality
\begin{equation}\label{th1-inequ2}
  \left|\fint_{B_{\rho^{k}R}(x_{0})}\pi\operatorname{d}\!x-\fint_{B_{R}(x_{0})}\pi\operatorname{d}\!x \right| \leq \rho^{-2}\sum_{i=0}^{k-1}\left|\fint_{B_{\rho^{i}R}(x_{0})}\left(\pi-\left(\pi\right)_{B_{\rho^{i}R}(x_{0})}\right)\operatorname{d}\!x \right|.
\end{equation}
holds for any $k\in\mathbb{N}$.

We first use the decay inequality \eqref{u-decayestimate} established in Lemma \ref{lem-u-decayestimate} to derive that
\begin{eqnarray}\label{th1-inequ3}
  && \sum_{i=1}^{k}\left[\fint_{B_{\rho^{i}\!R}(x_{0})}\left|D{\bf
u}-\left(D{\bf
u}\right)_{B_{\rho^{i}\!R}(x_{0})}\right|\operatorname{d}\!x
+\left(\fint_{B_{\rho^{i}\!R}(x_{0})}\left|{\pi}-\left(\pi\right)_{B_{\rho^{i}\!R}(x_{0})}\right|\operatorname{d}\!x
\right)^{\frac{p'}{p}}\right] \nonumber\\
&\leq &C_{1}\beta^{\frac{1}{p}}\sum_{i=0}^{k-1}\fint_{B_{\rho^{i}R}(x_{0})}\left|D{\bf
u}-\left(D {\bf
u}\right)_{B_{\rho^{i}R}(x_{0})}\right|\operatorname{d}\!x
 + C_{2}\sum_{i=0}^{k-1}\left(\fint_{B_{\rho^{i}R}(x_{0})}\left|{\bf F}-({\bf F})_{B_{\rho^{i}R}(x_{0})}\right|^{p'}\operatorname{d}\!x\right)^{\frac{1}{p}}\nonumber\\
&&+\,C_{3}\sum_{i=0}^{k-1}\left(\left[\mathcal{A}\right]_{\operatorname{ BMO}}^{\frac{\hat{\sigma}}{p}}(\rho^{i}R)+\varepsilon^{\frac{1}{p}}\right)\fint_{B_{\rho^{i}R}(x_{0})}\left(\mu+\left|D{\bf
u}\right|\right)\operatorname{d}\!x \nonumber\\
&\leq &\sum_{i=0}^{k-1}\left[C_{1}\beta^{\frac{1}{p}}+C_{3}\left(\left[\mathcal{A}\right]_{\operatorname{ BMO}}^{\frac{\hat{\sigma}}{p}}(\rho^{i}R)+\varepsilon^{\frac{1}{p}}\right)\right]\fint_{B_{\rho^{i}R}(x_{0})}\left(\mu+\left|D{\bf
u}-\left(D {\bf
u}\right)_{B_{\rho^{i}R}(x_{0})}\right|\right)\operatorname{d}\!x\nonumber\\
&&+\,C_{3}\sum_{i=0}^{k-1}\left(\left[\mathcal{A}\right]_{\operatorname{ BMO}}^{\frac{\hat{\sigma}}{p}}(\rho^{i}R)+\varepsilon^{\frac{1}{p}}\right)\left|\fint_{B_{\rho^{i}R}(x_{0})}D{\bf
u}\operatorname{d}\!x\right|\nonumber\\
&&+ C_{2}\sum_{i=0}^{k-1}\left(\fint_{B_{\rho^{i}R}(x_{0})}\left|{\bf F}-({\bf F})_{B_{\rho^{i}R}(x_{0})}\right|^{p'}\operatorname{d}\!x\right)^{\frac{1}{p}}.
\end{eqnarray}
We estimate the integral term involving $\bf F$ in \eqref{th1-inequ3} as follows
\begin{eqnarray*}
  && \sum_{i=0}^{k-1}\left(\fint_{B_{\rho^{i}R}(x_{0})}\left|{\bf F}-({\bf F})_{B_{\rho^{i}R}(x_{0})}\right|^{p'}\operatorname{d}\!x\right)^{\frac{1}{p}}\nonumber \\
   &\leq&  \left(\fint_{B_{R}(x_{0})}\left|{\bf F}-({\bf F})_{B_{R}(x_{0})}\right|^{p'}\operatorname{d}\!x\right)^{\frac{1}{p}}+ \sum_{i=1}^{\infty}\left(\fint_{B_{\rho^{i}R}(x_{0})}\left|{\bf F}-({\bf F})_{B_{\rho^{i}R}(x_{0})}\right|^{p'}\operatorname{d}\!x\right)^{\frac{1}{p}} \nonumber\\
   &=&\frac{1}{\ln2}\int_{R}^{2R}\left(\fint_{B_{R}(x_{0})}\left|{\bf F}-({\bf F})_{B_{R}(x_{0})}\right|^{p'}\operatorname{d}\!x\right)^{\frac{1}{p}}\frac{\operatorname{d}\!\varrho}{\varrho}\\
&&+\frac{1}{\ln\frac{1}{\rho}}
\sum_{i=1}^{\infty}\int_{\rho^{i}R}^{\rho^{i-1}R}\left(\fint_{B_{\rho^{i}R}(x_{0})}\left|{\bf F}-({\bf F})_{B_{\rho^{i}R}(x_{0})}\right|^{p'}\operatorname{d}\!x\right)^{\frac{1}{p}} \frac{\operatorname{d}\!\varrho}{\varrho} \nonumber\\
  &\leq& \frac{2^{\frac{2}{p}}}{\ln2}\int_{R}^{2R}\left(\fint_{B_{\varrho}(x_{0})}\left|{\bf F}-({\bf F})_{B_{\varrho}(x_{0})}\right|^{p'}\operatorname{d}\!x\right)^{\frac{1}{p}}\frac{\operatorname{d}\!\varrho}{\varrho}\\
  &&+
\frac{1}{\rho^{\frac{2}{p}}\ln\frac{1}{\rho}}
\sum_{i=1}^{\infty}\int_{\rho^{i}R}^{\rho^{i-1}R}\left(\fint_{B_{\varrho}(x_{0})}\left|{\bf F}-({\bf F})_{B_{\varrho}(x_{0})}\right|^{p'}\operatorname{d}\!x\right)^{\frac{1}{p}} \frac{\operatorname{d}\!\varrho}{\varrho} \nonumber\\
&\leq&C(p,\nu,L,\beta)\int_{0}^{2R}\left(\fint_{B_{\varrho}(x_{0})}\left|{\bf F}-({\bf F})_{B_{\varrho}(x_{0})}\right|^{p'}\operatorname{d}\!x\right)^{\frac{1}{p}}\frac{\operatorname{d}\!\varrho}{\varrho}.
\end{eqnarray*}
Then inserting the above estimate into \eqref{th1-inequ3}\,, and recalling that $\left[\mathcal{A}\right]_{\operatorname{ BMO}}(R)$ is a non-decreasing function with respect to $R$\,, one obtains
\begin{eqnarray}\label{th1-inequ4}
  && \sum_{i=1}^{k}\left[\fint_{B_{\rho^{i}\!R}(x_{0})}\left|D{\bf
u}-\left(D{\bf
u}\right)_{B_{\rho^{i}\!R}(x_{0})}\right|\operatorname{d}\!x
+\left(\fint_{B_{\rho^{i}\!R}(x_{0})}\left|{\pi}-\left(\pi\right)_{B_{\rho^{i}\!R}(x_{0})}\right|\operatorname{d}\!x
\right)^{\frac{p'}{p}}\right] \\
&\leq&\left[C_{1}\beta^{\frac{1}{p}}+C_{3}\left(\left[\mathcal{A}\right]_{\operatorname{ BMO}}^{\frac{\hat{\sigma}}{p}}(R)+\varepsilon^{\frac{1}{p}}\right)\right]\sum_{i=0}^{k-1}\fint_{B_{\rho^{i}R}(x_{0})}\left(\mu+\left|D{\bf
u}-\left(D {\bf
u}\right)_{B_{\rho^{i}R}(x_{0})}\right|\right)\operatorname{d}\!x\nonumber\\
&&+\,C_{3}\sum_{i=0}^{k-1}\left(\left[\mathcal{A}\right]_{\operatorname{ BMO}}^{\frac{\hat{\sigma}}{p}}(\rho^{i}R)+\varepsilon^{\frac{1}{p}}\right)\left|\fint_{B_{\rho^{i}R}(x_{0})}D{\bf
u}\operatorname{d}\!x\right|+ C_{2}\int_{0}^{2R}\left(\fint_{B_{\varrho}(x_{0})}\left|{\bf F}-({\bf F})_{B_{\varrho}(x_{0})}\right|^{p'}\operatorname{d}\!x\right)^{\frac{1}{p}}\frac{\operatorname{d}\!\varrho}{\varrho}\,,\nonumber
\end{eqnarray}
where $C_{1}=C_{1}(\varepsilon,\nu,L,p)$, $C_{2}=C_{2}(\varepsilon,\nu,L,p,\beta)$ and $C_{3}=C_{3}(\nu,L,p,\beta)$.
It is convenient for us to choose $\varepsilon^{\frac{1}{p}}=\left[\mathcal{A}\right]_{\operatorname{ BMO}}^{\frac{\hat{\sigma}}{p}}(R)$,
and then applying the monotonicity of $d(\cdot)$ to select the radius $R_{0}$ such that
\begin{eqnarray*}
  C_{3}\left[\mathcal{A}\right]_{\operatorname{ BMO}}^{\frac{\hat{\sigma}}{p}}(R)&=&\frac{C_{3}}{\ln 2}\int_{R}^{2R}\left[\mathcal{A}\right]_{\operatorname{ BMO}}^{\frac{\hat{\sigma}}{p}}(R)\frac{\operatorname{d}\!\varrho}{\varrho}\\
  &\leq&\frac{C_{3}}{\ln 2}\int_{0}^{2R}\left[\mathcal{A}\right]_{\operatorname{ BMO}}^{\frac{\hat{\sigma}}{p}}(\varrho)\frac{\operatorname{d}\!\varrho}{\varrho}\\
  &=&\frac{C_{3}}{\ln 2}d(2R)\leq\frac{C_{3}}{\ln 2}d(2R_{0})=\frac{1}{8}
\end{eqnarray*}
for any $R\leq R_{0}$. Furthermore, selecting $\beta$ sufficiently small such that
$C_{1}\beta^{\frac{1}{p}}\leq\frac{1}{4}$.
Thus, the first term on the right side of \eqref{th1-inequ4} can be absorbed by the left side that
\begin{eqnarray}\label{th1-inequ5}
  && \sum_{i=1}^{k}\left[\fint_{B_{\rho^{i}\!R}(x_{0})}\left|D{\bf
u}-\left(D{\bf
u}\right)_{B_{\rho^{i}\!R}(x_{0})}\right|\operatorname{d}\!x
+\left(\fint_{B_{\rho^{i}\!R}(x_{0})}\left|{\pi}-\left(\pi\right)_{B_{\rho^{i}\!R}(x_{0})}\right|\operatorname{d}\!x
\right)^{\frac{p'}{p}}\right]\nonumber\\
&\leq &\fint_{B_{R}(x_{0})}\left(\mu+\left|D{\bf
u}-\left(D {\bf
u}\right)_{B_{R}(x_{0})}\right|\right)\operatorname{d}\!x + C\int_{0}^{2R}\left(\fint_{B_{\varrho}(x_{0})}\left|{\bf F}-({\bf F})_{B_{\varrho}(x_{0})}\right|^{p'}\operatorname{d}\!x\right)^{\frac{1}{p}}\frac{\operatorname{d}\!\varrho}{\varrho}\nonumber\\
&&+\,C\sum_{i=0}^{k-1}\left[\mathcal{A}\right]_{\operatorname{ BMO}}^{\frac{\hat{\sigma}}{p}}(\rho^{i}R)\left|\fint_{B_{\rho^{i}R}(x_{0})}D{\bf
u}\operatorname{d}\!x\right|,
\end{eqnarray}
where $C=C(\nu,L,p,\left[\mathcal{A}\right]_{\operatorname{ BMO}}(\cdot))$.

Next, we turn our attention to estimate the last term on the right side of \eqref{th1-inequ5}. A combination of \eqref{th1-inequ5} and \eqref{eqn-minimal2} yields that
\begin{eqnarray}\label{th1-inequ6}
 && \left|\fint_{B_{\rho^{k+1}R}(x_{0})}D{\bf
u}\operatorname{d}\!x\right| \nonumber\\
  &=& \left|\sum_{i=0}^{k}\left(\fint_{B_{\rho^{i+1}R}(x_{0})}D{\bf
u}\operatorname{d}\!x-\fint_{B_{\rho^{i}R}(x_{0})}D{\bf
u}\operatorname{d}\!x\right)+\fint_{B_{R}(x_{0})}D{\bf
u}\operatorname{d}\!x\right| \nonumber \\
  &\leq& \rho^{-2}\sum_{i=0}^{k}\fint_{B_{\rho^{i}R}(x_{0})}\left|D{\bf
u}-\left(D{\bf u}\right)_{B_{\rho^{i}R}(x_{0})}\right|\operatorname{d}\!x+\left|\fint_{B_{R}(x_{0})}D{\bf
u}\operatorname{d}\!x\right| \nonumber\\
   &\leq& \frac{2}{\rho^{2}}\fint_{B_{R}(x_{0})}\left(\mu+\left|D{\bf
u}-\left(D {\bf
u}\right)_{B_{R}(x_{0})}\right|\right)\operatorname{d}\!x + C\int_{0}^{2R}\left(\fint_{B_{\varrho}(x_{0})}\left|{\bf F}-({\bf F})_{B_{\varrho}(x_{0})}\right|^{p'}\operatorname{d}\!x\right)^{\frac{1}{p}}\frac{\operatorname{d}\!\varrho}{\varrho}\nonumber\\
&&+\,C\sum_{i=0}^{k-1}\left[\mathcal{A}\right]_{\operatorname{ BMO}}^{\frac{\hat{\sigma}}{p}}(\rho^{i}R)\left|\fint_{B_{\rho^{i}R}(x_{0})}D{\bf
u}\operatorname{d}\!x\right|+\left|\fint_{B_{R}(x_{0})}D{\bf
u}\operatorname{d}\!x\right| \nonumber\\
&\leq& C\fint_{B_{R}(x_{0})}\left(\mu+\left|D{\bf
u}\right|\right)\operatorname{d}\!x + C\int_{0}^{2R}\left(\fint_{B_{\varrho}(x_{0})}\left|{\bf F}-({\bf F})_{B_{\varrho}(x_{0})}\right|^{p'}\operatorname{d}\!x\right)^{\frac{1}{p}}\frac{\operatorname{d}\!\varrho}{\varrho}\nonumber\\
&&+\,C\sum_{i=0}^{k-1}\left[\mathcal{A}\right]_{\operatorname{ BMO}}^{\frac{\hat{\sigma}}{p}}(\rho^{i}R)\left|\fint_{B_{\rho^{i}R}(x_{0})} D{\bf
u}\operatorname{d}\!x\right|.
\end{eqnarray}

Setting the notation
\begin{equation*}
  M:=\fint_{B_{R}(x_{0})}\left(\mu+\left|D{\bf
u}\right|\right)\operatorname{d}\!x + \int_{0}^{2R}\left(\fint_{B_{\varrho}(x_{0})}\left|{\bf F}-({\bf F})_{B_{\varrho}(x_{0})}\right|^{p'}\operatorname{d}\!x\right)^{\frac{1}{p}}\frac{\operatorname{d}\!\varrho}{\varrho}\,,
\end{equation*}
we claim that the following uniform estimate
\begin{equation}\label{th1-inequ7}
  \left|\fint_{B_{\rho^{k+1}R}(x_{0})}D{\bf
u}\operatorname{d}\!x\right|\leq C(p,\nu,L)M
\end{equation}
holds for any $k\in\mathbb{N}$.

The proof of this claim is based on the method of induction. Obviously, the estimate for the case of $k=0$ is trivial. We assume that \eqref{th1-inequ7} is true for all $k\leq k_{0}$, and then verify the case of $k=k_{0}+1$.
In terms of \eqref{th1-inequ6}, we get
\begin{eqnarray}\label{th1-inequ10}
  \left|\fint_{B_{\rho^{k_{0}+2}R}(x_{0})}D{\bf
u}\operatorname{d}\!x\right|&\leq& CM+ C\sum_{i=0}^{k_{0}}\left[\mathcal{A}\right]_{\operatorname{ BMO}}^{\frac{\hat{\sigma}}{p}}(\rho^{i}R)\left|\fint_{B_{\rho^{i}R}(x_{0})}D{\bf
u}\operatorname{d}\!x\right| \nonumber\\
   &\leq& CM+CM\sum_{i=0}^{k_{0}}\left[\mathcal{A}\right]_{\operatorname{ BMO}}^{\frac{\hat{\sigma}}{p}}(\rho^{i}R).
\end{eqnarray}
Applying the fact that $\left[\mathcal{A}\right]_{\operatorname{BMO}}(\cdot)$ is non-deceasing, $\rho\in\left(0,\frac{1}{4}\right]$ and the definition of $d(\cdot)$ in \eqref{dini-bmo}, we obtain that
\begin{eqnarray*}
\sum_{i=0}^{k_{0}}\left[\mathcal{A}\right]_{\operatorname{ BMO}}^{\frac{\hat{\sigma}}{p}}(\rho^{i}R)
&\leq& \sum_{i=0}^{\infty}\left[\mathcal{A}\right]_{\operatorname{ BMO}}^{\frac{\hat{\sigma}}{p}}(\rho^{i}R) \nonumber \\
   &=& \frac{1}{\ln 2}\int_{R}^{2R}\left(\left[\mathcal{A}\right]_{\operatorname{ BMO}}^{\frac{\hat{\sigma}}{p}}(R)\right)\frac{\operatorname{d}\!\varrho}{\varrho}+\frac{1}{\ln \frac{1}{\rho}}\sum_{i=1}^{\infty}\int_{\rho^{i}R}^{\rho^{i-1}R}\left(\left[\mathcal{A}\right]_{\operatorname{ BMO}}^{\frac{\hat{\sigma}}{p}}(\sigma^{i}R)\right)\frac{\operatorname{d}\!\varrho}{\varrho} \nonumber \\
   &\leq& \frac{1}{\ln 2}\int_{R}^{2R}\left(\left[\mathcal{A}\right]_{\operatorname{ BMO}}^{\frac{\hat{\sigma}}{p}}(\varrho)\right)\frac{\operatorname{d}\!\varrho}{\varrho}+\frac{1}{\ln \frac{1}{\rho}}\sum_{i=1}^{\infty}\int_{\rho^{i}R}^{\rho^{i-1}R}\left(\left[\mathcal{A}\right]_{\operatorname{ BMO}}^{\frac{\hat{\sigma}}{p}}(\varrho)\right)\frac{\operatorname{d}\!\varrho}{\varrho} \nonumber \\
   &\leq& \left(\frac{1}{\ln 2}+\frac{1}{\ln 4}\right)\int_{0}^{2R}\left(\left[\mathcal{A}\right]_{\operatorname{ BMO}}^{\frac{\hat{\sigma}}{p}}(\varrho)\right)\frac{\operatorname{d}\!\varrho}{\varrho} \nonumber\\
   &\leq&\frac{2}{\ln 2}\,d\left(2R\right).
\end{eqnarray*}
We further restrict the value of $R_{0}$ such that
\begin{equation*}
  \frac{2}{\ln 2}\,d\left(2R_{0}\right)\leq1\,.
\end{equation*}
By virtue of the monotonicity of $d(\cdot)$\,, we have
\begin{equation}\label{th1-inequ8}
  \sum_{i=0}^{k_{0}}\left[\mathcal{A}\right]_{\operatorname{ BMO}}^{\frac{\hat{\sigma}}{p}}(\rho^{i}R)\leq1
\end{equation}
for any $R\leq R_{0}$. 
Thus the above argument implies that
\begin{equation*}
  \left|\fint_{B_{\rho^{k_{0}+2}R}(x_{0})}D{\bf
u}\operatorname{d}\!x\right|\leq CM\,.
\end{equation*}
Hence, the assertion \eqref{th1-inequ7} holds whenever $k\in\mathbb{N}$.

In the sequel, by applying this claim to \eqref{th1-inequ5}, and combining \eqref{th1-inequ8} with \eqref{th1-inequ11} and \eqref{eqn-minimal2}, we deduce that
\begin{eqnarray}\label{th1-inequ12}
 && \sum_{i=1}^{k}\left[\fint_{B_{\rho^{i}\!R}(x_{0})}\left|D{\bf
u}-\left(D{\bf
u}\right)_{B_{\rho^{i}\!R}(x_{0})}\right|\operatorname{d}\!x
+\left(\fint_{B_{\rho^{i}\!R}(x_{0})}\left|{\pi}-\left(\pi\right)_{B_{\rho^{i}\!R}(x_{0})}\right|\operatorname{d}\!x
\right)^{\frac{p'}{p}}\right]\nonumber\\
&\leq &C\fint_{B_{R}(x_{0})}\left(\mu+\left|D{\bf
u}\right|\right)\operatorname{d}\!x + C\int_{0}^{2R}\left(\fint_{B_{\varrho}(x_{0})}\left|{\bf F}-({\bf F})_{B_{\varrho}(x_{0})}\right|^{p'}\operatorname{d}\!x\right)^{\frac{1}{p}}\frac{\operatorname{d}\!\varrho}{\varrho}\,,
\end{eqnarray}
where $C$ depends only on $p$, $\nu$, $L$ and $\left[\mathcal{A}\right]_{\operatorname{BMO}}(\cdot)$. We pass to the limit as $k\rightarrow\infty$ in \eqref{th1-inequ12} to derive
\begin{eqnarray}\label{th1-inequ9}
   && \sum_{i=1}^{\infty}\left[\fint_{B_{\rho^{i}\!R}(x_{0})}\left|D{\bf
u}-\left(D{\bf
u}\right)_{B_{\rho^{i}\!R}(x_{0})}\right|\operatorname{d}\!x
+\left(\fint_{B_{\rho^{i}\!R}(x_{0})}\left|{\pi}-\left(\pi\right)_{B_{\rho^{i}\!R}(x_{0})}\right|\operatorname{d}\!x
\right)^{\frac{p'}{p}}\right]\nonumber\\
&\leq &C\fint_{B_{R}(x_{0})}\left(\mu+\left|D{\bf
u}\right|\right)\operatorname{d}\!x + C\int_{0}^{2R}\left(\fint_{B_{\varrho}(x_{0})}\left|{\bf F}-({\bf F})_{B_{\varrho}(x_{0})}\right|^{p'}\operatorname{d}\!x\right)^{\frac{1}{p}}\frac{\operatorname{d}\!\varrho}{\varrho}\,.
\end{eqnarray}
And then, let $k\rightarrow\infty$ in \eqref{th1-inequ1} and \eqref{th1-inequ2}, we combine the Lebesgue differentiation Theorem with \eqref{th1-inequ9} to conclude that
\begin{eqnarray*}
   &&  \left|D{\bf u}(x_{0})-\fint_{B_{R}(x_{0})}D{\bf u}\operatorname{d}\!x \right|+ \left|\pi(x_{0})-\fint_{B_{R}(x_{0})}\pi\operatorname{d}\!x \right|^{\frac{p'}{p}} \\
   &\leq& \rho^{-2}\sum_{i=0}^{\infty}\left|\fint_{B_{\rho^{i}\!R}(x_{0})}\!\left(D{\bf
u}-\left(D{\bf
u}\right)_{B_{\sigma^{i}\!R}(x_{0})}\right)\operatorname{d}\!x\right|+\rho^{-\frac{2p'}{p}}C(p)
\sum_{i=0}^{\infty}\left|\fint_{B_{\rho^{i}\!R}(x_{0})}\!\left({\pi}-\left(\pi\right)_{B_{\rho^{i}\!R}(x_{0})}
\right)\operatorname{d}\!x\right|^{\frac{p'}{p}}  \\
  &\leq& C\fint_{B_{R}(x_{0})}\left(\mu+\left|D{\bf
u}\right|\right)\operatorname{d}\!x +C\left(\fint_{B_{R}(x_{0})}\left|\pi\right|\operatorname{d}\!x\right)^{\frac{p'}{p}} + C\int_{0}^{2R}\left(\fint_{B_{\varrho}(x_{0})}\left|{\bf F}-({\bf F})_{B_{\varrho}(x_{0})}\right|^{p'}\operatorname{d}\!x\right)^{\frac{1}{p}}\frac{\operatorname{d}\!\varrho}{\varrho}
\end{eqnarray*}
for almost every $x_{0}\in\Omega$. As a consequence of the above inequalities, we finally derive
the following pointwise estimate
\begin{eqnarray*}
&& |D{\bf u}(x_{0})|+|\pi (x_{0})|^{\frac{p'}{p}} \\
&\leq& C\fint_{B_{R}(x_{0})}\left(\mu+\left|D{\bf u}\right|\right)\operatorname{d}\!x+ C\left(\fint_{B_{R}(x_{0})}\left|\pi\right|\operatorname{d}\!x\right)^{\frac{p'}{p}}+C\int_{0}^{2R}\left(\fint_{B_{\varrho}(x_{0})}\left|{\bf F}-({\bf F})_{B_{\varrho}(x_{0})}\right|^{p'}\operatorname{d}\!x\right)^{\frac{1}{p}}\frac{\operatorname{d}\!\varrho}{\varrho}
\end{eqnarray*}
holds for almost every $x_{0}\in\Omega$ and every $B_{2R}(x_{0})\subset\Omega$ with $R\leq R_{0}$,
where the positive constant $C=C(\nu,L,p,\left[\mathcal{A}\right]_{\operatorname{BMO}}(\cdot))$ and the radius $R_{0}=R_{0}\left(\nu,L,p,d(\cdot)\right)$.
Thus we complete the proof of Theorem \ref{Th1}\,.
\end{proof}

Subsequently, it is devoted to establish a nonlinear potential estimate for the weak solution to \eqref{model} with $p\geq2$ in higher dimensions. Note that there is no extra regularity assumption on the partial map $x\mapsto\mathcal{A}(x,\cdot)$ from now on.

\begin{proof}[Proof of Theorem \ref{Th2}] By proceeding similarly as in the proof of Lemma \ref{caccioppoli-u}\,, we derive the following Caccioppoli inequality
\begin{equation}\label{th2-inequ1}
  \fint_{B_{R}(x_{0})}|D{\bf v}|^{p}\operatorname{d}\!x\leq\frac{C}{R^{p}}\fint_{B_{2R}(x_{0})}|{\bf v}-\left(\bf v\right)_{B_{2R}(x_{0})}|^{p}\operatorname{d}\!x+C\mu^{p}\,,
\end{equation}
where $C=C(n,p,\nu,L)$.
Applying the Sobolev-Poincar\'{e} inequality introduced in \cite[Theorem 7]{DE} and combining with Korn's inequality \eqref{Korn2}, H\"{o}lder's inequality, \eqref{eqn-minimal2} and \eqref{th2-inequ1}, one can find a weak solution pair $\left({\bf v}, \pi_{\bf v}\right)$ to \eqref{ComparisonSystem} with $\left(\bf Wv\right)_{B_{R}(x_{0})}=0$ such that
\begin{eqnarray}\label{th2-inequ4}
&& \left(\fint_{B_{R}(x_{0})}\left|{\bf
v}-\left({\bf
v}\right)_{B_{R}(x_{0})}\right|^{s}\operatorname{d}\!x\right)^{\frac{1}{s}} \nonumber\\
&\leq& C(n,p)R\left(\fint_{B_{R}(x_{0})}\left|\nabla{\bf v}\right|^{p}\operatorname{d}\!x\right)^{\frac{1}{p}} \nonumber\\
&\leq& CR\left[\,\left(\fint_{B_{R}(x_{0})}\left|\nabla{\bf v}-\left(\nabla{\bf
v}\right)_{B_{R}(x_{0})}\right|^{p}\operatorname{d}\!x\right)^{\frac{1}{p}}+\left|\fint_{B_{R}(x_{0})}\nabla{\bf v}\operatorname{d}\!x\right|\,\right] \nonumber\\
&=& CR\left[\,\left(\fint_{B_{R}(x_{0})}\left|\nabla{\bf v}-\left(\nabla{\bf
v}\right)_{B_{R}(x_{0})}\right|^{p}\operatorname{d}\!x\right)^{\frac{1}{p}}+\left|\fint_{B_{R}(x_{0})}D{\bf v}\operatorname{d}\!x\right|\,\right]\nonumber \\
&\leq& CR\left(\fint_{B_{R}(x_{0})}\left|D{\bf v}\right|^{p}\operatorname{d}\!x\right)^{\frac{1}{p}}\nonumber \\
&\leq&C\left(\fint_{B_{2R}(x_{0})}\left|{\bf
v}-\left({\bf
v}\right)_{B_{2R}(x_{0})}\right|^{p}\operatorname{d}\!x\right)^{\frac{1}{p}}+CR\mu
\end{eqnarray}
for some $s>p$, where ${\bf Wv}=\frac{\nabla {\bf v}-\left(\nabla {\bf v}\right)^{T}}{2}$ and $C=C(n,p,\nu,L)$.
Thus, it follows from Lemma \ref{reverse-holder} and H\"{o}lder's inequality that
\begin{equation}\label{th2-inequ2}
  \left(\fint_{B_{R}(x_{0})}\left|{\bf
v}-\left({\bf
v}\right)_{B_{R}(x_{0})}\right|^{p}\operatorname{d}\!x\right)^{\frac{1}{p}}\leq C\fint_{B_{2R}(x_{0})}\left|{\bf
v}-\left({\bf
v}\right)_{B_{2R}(x_{0})}\right|\operatorname{d}\!x+CR\mu\,.
\end{equation}

Furthermore, we claim that the following Campanato-type decay estimate for $\bf v$ holds
\begin{equation}\label{th2-inequ3}
  \left(\fint_{B_{\upsilon R}(x_{0})}\left|{\bf
v}-\left({\bf
v}\right)_{B_{\upsilon R}(x_{0})}\right|^{p}\operatorname{d}\!x\right)^{\frac{1}{p}}\leq C\upsilon^{\gamma}\left(\fint_{B_{R}(x_{0})}\left|{\bf
v}-\left({\bf
v}\right)_{B_{R}(x_{0})}\right|^{p}\operatorname{d}\!x\right)^{\frac{1}{p}}+CR\mu
\end{equation}
for any $\upsilon\in (0,1]$ and some $\gamma=\gamma(n,p,\nu,L)>0$, where $C=C(n,\lambda,\Lambda,p)$.
In fact, we only need to show that it holds for $\upsilon\in (0,\frac{1}{4})$, since the analogous inequality for $\upsilon\in [\frac{1}{4},1]$ is trivial. Now using the same argument as in the estimates of \eqref{th2-inequ4} and combining with the higher integrability result \eqref{v-reverseholder} for $D{\bf v}$ to yield that
\begin{eqnarray*}
 \left(\fint_{B_{\upsilon R}(x_{0})}\left|{\bf
v}-\left({\bf
v}\right)_{B_{\upsilon R}(x_{0})}\right|^{p}\operatorname{d}\!x\right)^{\frac{1}{p}}
&\leq& C(n,p)\upsilon R\left(\fint_{B_{\upsilon R}(x_{0})}\left|\nabla{\bf v}\right|^{p}\operatorname{d}\!x\right)^{\frac{1}{p}}\\
&\leq& C\upsilon R\left(\fint_{B_{\upsilon R}(x_{0})}\left|D{\bf v}\right|^{p}\operatorname{d}\!x\right)^{\frac{1}{p}}\\
&\leq& C\upsilon^{1-\frac{n}{\theta}} R\left(\fint_{B_{\frac{1}{4} R}(x_{0})}\left|D{\bf v}\right|^{\theta}\operatorname{d}\!x\right)^{\frac{1}{\theta}}\\
&\leq& C\upsilon^{1-\frac{n}{\theta}} R\left(\fint_{B_{\frac{1}{2} R}(x_{0})}\left|D{\bf v}\right|^{p}\operatorname{d}\!x\right)^{\frac{1}{p}}\\
&\leq&C\upsilon^{1-\frac{n}{\theta}}\left(\fint_{B_{R}(x_{0})}\left|{\bf
v}-\left({\bf
v}\right)_{B_{R}(x_{0})}\right|^{p}\operatorname{d}\!x\right)^{\frac{1}{p}}+CR\mu.
\end{eqnarray*}
Then the desired estimate \eqref{th2-inequ3} holds for $\upsilon\in (0,\frac{1}{4})$, which is ensured by $p\leq n<\theta$.

In the sequel, our intention is to derive the Campanato type decay estimate for the weak solution $\bf u$ to \eqref{model}. Applying H\"{o}lder's inequality, \eqref{eqn-minimal2}, \eqref{u-v}, \eqref{th2-inequ3} and \eqref{th2-inequ2}, we conclude that
\begin{eqnarray*}
  && \left(\fint_{B_{\upsilon\!R}(x_{0})}\left|{\bf
u}-\left({\bf
u}\right)_{B_{\upsilon\!R}(x_{0})}\right|\operatorname{d}\!x\right)^{p}\\
   &\leq& C(p)\fint_{B_{\upsilon\!R}(x_{0})}\left|{\bf
u}-\left({\bf
v}\right)_{B_{\upsilon\!R}(x_{0})}\right|^{p}\operatorname{d}\!x \\
  &\leq& C(n,p)\upsilon^{-n}\fint_{B_{R}(x_{0})}\left|{\bf
u}-{\bf
v}\right|^{p}\operatorname{d}\!x+C(p)\fint_{B_{\upsilon\!R}(x_{0})}\left|{\bf
v}-\left({\bf
v}\right)_{B_{\upsilon\!R}(x_{0})}\right|^{p}\operatorname{d}\!x \\
   &\leq& C\upsilon^{-n}R^{p}\fint_{B_{2R}(x_{0})}\!\left(\mu^{p}+\left|{\bf F}-({\bf F})_{B_{2R}(x_{0})}\right|^{p'}\right)\operatorname{d}\!x+ C\upsilon^{p\gamma}\fint_{B_{R}(x_{0})}\left|{\bf
v}-\left({\bf
v}\right)_{B_{R}(x_{0})}\right|^{p}\operatorname{d}\!x \\
&\leq&  C\upsilon^{-n}R^{p}\fint_{B_{2R}(x_{0})}\!\left(\mu^{p}+\left|{\bf F}-({\bf F})_{B_{2R}(x_{0})}\right|^{p'}\right)\operatorname{d}\!x+ C\upsilon^{p\gamma}\left(\fint_{B_{R}(x_{0})}\left(\left|{\bf
u}-{\bf
v}\right|+\left|{\bf
u}-\left({\bf
u}\right)_{B_{2R}(x_{0})}\right|\right)\operatorname{d}\!x\right)^{p} \\
&\leq& C\upsilon^{p\gamma}\left(\fint_{B_{2R}(x_{0})}\left|{\bf
u}-\left({\bf
u}\right)_{B_{2R}(x_{0})}\right|\operatorname{d}\!x\right)^{p}+ C_{\upsilon}R^{p}\fint_{B_{2R}(x_{0})}\!\left(\mu^{p}+\left|{\bf F}-({\bf F})_{B_{2R}(x_{0})}\right|^{p'}\right)\operatorname{d}\!x
\end{eqnarray*}
for $p\geq2$, here $C=C(n,p,\nu,L)$ and $C_{\upsilon}=C(n,p,\nu,L,\upsilon)$. Then we conclude that
\begin{equation*}
  \fint_{B_{\upsilon\!R}(x_{0})}\left|{\bf
u}-\left({\bf
u}\right)_{B_{\upsilon\!R}(x_{0})}\right|\operatorname{d}\!x\leq C\upsilon^{\gamma}\fint_{B_{2R}(x_{0})}\left|{\bf
u}-\left({\bf
u}\right)_{B_{2R}(x_{0})}\right|\operatorname{d}\!x+ C_{\upsilon}R\left(\fint_{B_{2R}(x_{0})}\!\left(\mu^{p}+\left|{\bf F}\right|^{p'}\right)\operatorname{d}\!x\right)^{\frac{1}{p}}.
\end{equation*}
The subsequent proof goes exactly as that of Theorem \ref{Th1}\,. It suffices to reestimate the integral term involving $\bf F$ as follows
\begin{eqnarray*}
  && \sum_{i=0}^{k-1}\upsilon^{i}R\left(\fint_{B_{\upsilon^{i}R}(x_{0})}\left(\mu^{p}+\left|{\bf
F}\right|^{p'}\right)\operatorname{d}\!x\right)^{\frac{1}{p}}\nonumber \\
   &\leq&  R\left(\fint_{B_{R}(x_{0})}\left(\mu^{p}+\left|{\bf
F}\right|^{p'}\right)\operatorname{d}\!x\right)^{\frac{1}{p}}+ \sum_{i=1}^{\infty}\upsilon^{i}R\left(\fint_{B_{\upsilon^{i}R}(x_{0})}\left(\mu^{p}+\left|{\bf
F}\right|^{p'}\right)\operatorname{d}\!x\right)^{\frac{1}{p}} \nonumber\\
   &=&\frac{1}{\ln2}\int_{R}^{2R}R\left(\fint_{B_{R}(x_{0})}\left(\mu^{p}+\left|{\bf
F}\right|^{p'}\right)\operatorname{d}\!x\right)^{\frac{1}{p}}\frac{\operatorname{d}\!\varrho}{\varrho}
+\frac{1}{\ln\frac{1}{\upsilon}}
\sum_{i=1}^{\infty}\int_{\upsilon^{i}R}^{\upsilon^{i-1}R}\upsilon^{i}R\left(\fint_{B_{\upsilon^{i}R}(x_{0})}
\left(\mu^{p}+\left|{\bf
F}\right|^{p'}\right)\operatorname{d}\!x\right)^{\frac{1}{p}} \frac{\operatorname{d}\!\varrho}{\varrho} \nonumber\\
  &\leq& \frac{2^{\frac{n}{p}}}{\ln2}\int_{R}^{2R}\varrho\left(\fint_{B_{\varrho}(x_{0})}\left(\mu^{p}+\left|{\bf
F}\right|^{p'}\right)\operatorname{d}\!x\right)^{\frac{1}{p}}\frac{\operatorname{d}\!\varrho}{\varrho}
+\frac{1}{\upsilon^{\frac{n}{p}}\ln\frac{1}{\upsilon}}
\sum_{i=1}^{\infty}\int_{\upsilon^{i}R}^{\upsilon^{i-1}R}\varrho\left(\fint_{B_{\varrho}(x_{0})}\left(\mu^{p}+\left|{\bf
F}\right|^{p'}\right)\operatorname{d}\!x\right)^{\frac{1}{p}} \frac{\operatorname{d}\!\varrho}{\varrho} \nonumber\\
&\leq&C(n,p,\nu,L)\int_{0}^{2R}\left( \fint_{B_{\varrho}(x_{0})}\left(\mu^{p}+\left|{\bf F}\right|
^{p'}\right)\operatorname{d}\!x\right)^{\frac{1}{p}}\operatorname{d}\!\varrho\nonumber\\
&=&C\,{\bf W}_{\frac{p}{p+1},p+1}^{2R}\left(\mu^{p}+\left|{\bf F}\right|
^{p'}\right)(x_{0})\,.
\end{eqnarray*}
Therefore, we deduce the desired zero order pointwise estimate \eqref{zero-estimate}, which completes the proof of Theorem \ref{Th2}\,.
\end{proof}

\section*{Acknowledgments}The authors are very grateful to Professor G. Mingione for suggesting this interesting problem  to us.
The authors are supported by the National Natural Science Foundation of China (NNSF Grant No.~12071229 and No.~12001333), and Shandong Provincial Natural Science Foundation (Grant No. ZR2020QA005).

\bibliography{bibliography}

\end{document}